\documentclass{elsarticle}
\usepackage{amssymb}
\usepackage{amsfonts}
\usepackage{amsmath}
\usepackage{stmaryrd}

\begin{document}

\newtheorem{theorem}{Theorem}
\newtheorem{definition}[theorem]{Definition}
\newtheorem{lemma}[theorem]{Lemma}
\newtheorem{proposition}[theorem]{Proposition}
\newtheorem{example}[theorem]{Example}
\newtheorem{corollary}[theorem]{Corollary}
\newtheorem{conjecture}[theorem]{Conjecture}
\newenvironment{note}{\noindent{\bf Note:}}{\medskip}
\newenvironment{proof}{\noindent{\bf Proof:}}{\medskip}
\def\squareforqed{\hbox{\rlap{$\sqcap$}$\sqcup$}}
\def\qed{\ifmmode\squareforqed\else{\unskip\nobreak\hfil
\penalty50\hskip1em\null\nobreak\hfil\squareforqed
\parfillskip=0pt\finalhyphendemerits=0\endgraf}\fi}

\def\cal#1{{\mathcal #1}}
\def\bb#1{{\mathbb #1}}
\def\bsigma#1{\mathbf{\Sigma}^0_{#1}}
\def\bpi#1{\mathbf{\Pi}^0_{#1}}
\def\bdelta#1{\mathbf{\Delta}^0_{#1}}
\def\analytic{\mathbf{\Sigma}^1_1}
\def\coanalytic{\mathbf{\Pi}^1_1}
\def\bianalytic{\mathbf{\Delta}^1_1}
\def\baire{\omega^{\omega}}
\def\cntbased{$\omega${\bf Top}$_0$}
\def\bdiff#1{\mathbf{\Sigma}^{-1}_{#1}}

\def\wayabovearrow{\rlap{\raise-.25ex\hbox{$\shortuparrow$}}\raise.25ex\hbox{$\shortuparrow$}}
\def\waybelowarrow{\rlap{\raise.25ex\hbox{$\shortdownarrow$}}\raise-.25ex\hbox{$\shortdownarrow$}}

\def\abs#1{{\lvert#1\rvert}}


\begin{frontmatter}

\title{Quasi-Polish Spaces}

\author{Matthew de Brecht}
\ead{matthew@nict.go.jp}

\address{National Institute of Information and Communications Technology\\ Kyoto, Japan}

\begin{abstract}
We investigate some basic descriptive set theory for countably based completely quasi-metrizable topological spaces, which we refer to as quasi-Polish spaces. These spaces naturally generalize much of the classical descriptive set theory of Polish spaces to the non-Hausdorff setting. We show that a subspace of a quasi-Polish space is quasi-Polish if and only if it is $\bpi 2$ in the Borel hierarchy. Quasi-Polish spaces can be characterized within the framework of Type-2 Theory of Effectivity as precisely the countably based spaces that have an admissible representation with a Polish domain. They can also be characterized domain theoretically as precisely the spaces that are homeomorphic to the subspace of all non-compact elements of an $\omega$-continuous domain. Every countably based locally compact sober space is quasi-Polish, hence every $\omega$-continuous domain is quasi-Polish. A metrizable space is quasi-Polish if and only if it is Polish. We show that the Borel hierarchy on an uncountable quasi-Polish space does not collapse, and that the Hausdorff-Kuratowski theorem generalizes to all quasi-Polish spaces.
\end{abstract}


\end{frontmatter}


\section{Introduction}

Separable completely metrizable spaces, called Polish spaces, are perhaps the best understood and most widely researched class of topological spaces. These include the space of natural numbers with the discrete topology, the real numbers with the Euclidean topology, as well as the separable Hilbert and Banach spaces.

Descriptive set theory \cite{kechris} has proven to be an invaluable tool for the study of Polish spaces, from providing elegant characterizations of Polish spaces and a means of quantifying the complexity of ``definable'' sets, to exploring the limits of what is provable within Zermelo-Fraenkel set theory with the axiom of choice. The techniques of descriptive set theory have been successfully applied to many fields such as functional analysis, topological group theory, and mathematical logic.

Somewhat more recently, however, there has been growing interest in non-metrizable spaces, in particular the continuous lattices and domains of domain theory \cite{etal_scott}. These spaces generally fail to satisfy even the $T_1$-separation axiom, but naturally occur in the general theory of computation, the analysis of function spaces, as well as in algebra and logic. Continuous domains are also characterized by a kind of completeness property, which at first glance seems rather different than the completeness property of a metric.

Another interest in non-metrizable spaces comes from the theory of quasi-metrics \cite{kunzi_intro} and partial metrics \cite{Matthews}. These are generalizations of metrics, where quasi-metrics are the result of removing the axiom of symmetry, and partial metrics are the result of removing the requirement that the distance from a point to itself be zero. These generalized metrics provide a useful tool for defining and analyzing non-metrizable topologies, and have applications in general topology, theoretical computer science, and other fields of mathematics.


Despite the great success of descriptive set theory with the analysis of metrizable spaces, the extension of this approach to more general spaces seems to have been largely overlooked by the mathematical community. This is most likely due to the fact that the traditional definition of the Borel hierarchy in terms of $F_\sigma$ and $G_\delta$ sets behaves poorly on non-metrizable spaces. However, a hint to overcoming this technicality was given at least as early as 1976, in the perhaps not so well known Section 6 of a very well known paper by Dana Scott \cite{scott_datatypes}. There it was shown that countable intersections of \emph{boolean combinations} of open sets, called $\cal B_\delta$ sets, can be much more interesting than $G_\delta$ sets when dealing with non-metrizable spaces like domains. This research was continued briefly by A. Tang \cite{tang1979, tang1981}, but the focus of these papers was exclusively on the space $\cal P(\omega)$, the power set of the natural numbers with the Scott-topology, and dealt mainly with $\cal B_\delta$ sets and their complements, the $\cal B_\sigma$ sets.

Victor Selivanov was the first to investigate the Borel hierarchy systematically on general topological spaces, using a modified version of the hierarchy that identifies the $\cal B_\sigma$ and $\cal B_\delta$ sets of Scott and Tang with the levels $\bsigma 2$ and $\bpi 2$. Selivanov demonstrated the viability of this approach by showing that many basic theorems of descriptive set theory extend to the class of $\omega$-continuous domains, and made many other contributions such as studying the difference hierarchy and Wadge reducibility on domains (see \cite{selivanov} and the references therein for an overview). Despite these many successes, descriptive set theory became divided into the traditional theory for Polish spaces and the recently emerging theory for $\omega$-continuous domains. Selivanov \cite{selivanov2008} then posed the question as to whether or not there exists a more general class of spaces, containing both Polish spaces and $\omega$-continuous domains, which can allow a unified generalization of the descriptive set theory of Polish spaces.

The goal of this paper is to introduce a class of spaces, which we call quasi-Polish spaces, which we propose as a solution to the question posed by Selivanov. In a sense these spaces are not new, they are defined as the countably based spaces which admit a (Smyth)-complete quasi-metric, and correspond (at least up to homeomorphism) to the $\cal B_\delta$ subspaces of $\cal P(\omega)$ investigated by Scott and Tang. However, to our knowledge this paper is the first attempt to develop a coherent descriptive set theory for these spaces, as well as demonstrate both their generality and nice completeness properties. For example, we will see that the class of quasi-Polish spaces is general enough to contain both the Polish spaces and the countably based locally compact sober spaces, hence all $\omega$-continuous domains, but is not too general as demonstrated by the fact that every quasi-Polish space is sober and every metrizable quasi-Polish space is Polish. 

The majority of this paper will be dedicated to showing the naturalness of extending the descriptive set theory of Polish spaces to the class of quasi-Polish spaces. For example, a subspace of a quasi-Polish space is quasi-Polish if and only if it is a $\bpi 2$ subset, and quasi-Polish spaces have a game-theoretic characterization in terms of a simple modification of the strong Choquet game. The topology of quasi-Polish spaces can also be extended to finer quasi-Polish topologies in a manner similar to the case for Polish spaces. The naturalness of quasi-Polish spaces will also be demonstrated by showing that they are precisely the spaces that are homeomorphic to the subspace of non-compact elements of an $\omega$-continuous domain, and that they are precisely the countably based spaces that have a total admissible representation in the sense of Type 2 Theory of Effectivity.

In addition to our multiple characterizations of the countably based spaces that admit complete quasi-metrics, we will also provide solutions to several other problems that, to our knowledge, remain open. In particular, we will combine our techniques with the work of H. Junnila and H.-P. A. K\"{u}nzi \cite{junnila_kunzi} to provide a complete characterization of the countably based spaces which admit a bicomplete quasi-metric in terms of level $\bpi 3$ of the Borel hierarchy. We will also show that the quasi-Polish spaces satisfying the $T_1$-axiom provide a solution to the problem posed by K. Martin \cite{keye_martin} of characterizing the spaces that can be modeled by an $\omega$-ideal domain. We will also extend the results of Selivanov by proving the Hausdorff-Kuratowski theorem for quasi-Polish spaces in full generality.

The basic outline of this paper is as follows. We will introduce basic definitions and notation in the following section. Section \ref{sec:borelhierarchy} defines the Borel hierarchy for general topological spaces. Quasi-metrics are defined in Section \ref{sec:quasimetricspaces}, and quasi-Polish spaces are defined and characterized in Section \ref{sec:complete_quasimetricspaces}. Bicomplete quasi-metrics and complete partial metrics are briefly discussed in Section \ref{sec:othercompletegeneralmetrics}, where we provide a characterization of countably based bicompletely quasi-metrizable spaces. Sections \ref{sec:opencontsurjections} through \ref{sec:scatteredspaces} investigate general properties of quasi-Polish spaces, provide alternative characterizations, and demonstrate that many important classes of spaces are quasi-Polish. Section \ref{sec:HausdorffKuratowski} extends the Hausdorff-Kuratowski theorem to quasi-Polish spaces, and Section \ref{sec:topextension} investigates extensions of quasi-Polish topologies.


\section{Preliminaries}

We will assume that the reader is familiar with general topology. Ideally, the reader will also be familiar with the classical descriptive set theory of Polish spaces and have a basic understanding of domain theory, and we will only provide some of the basic definitions in this section. A reader familiar with both of these fields may feel free to skip this section and only return to it later if necessary. Our main reference for descriptive set theory is \cite{kechris}, and our main reference for domain theory is \cite{etal_scott}.

We use $\omega$ to denote the set of natural numbers, and $\baire$ to denote the set of functions on $\omega$. Finite sequences of natural numbers will be denoted by $\omega^{<\omega}$. For any $\sigma\in\omega^{<\omega}$, we write $\abs{\sigma}$ for the length of $\sigma$, and write $\sigma\diamond i$ for the sequence obtained by appending $i\in\omega$ to the end of $\sigma$. The prefix relation and strict prefix relation on $\omega^{<\omega}$ will be denoted by $\preceq$ and $\prec$, respectively. For $\sigma\in\omega^{<\omega}$ and $p\in\baire$, we will write $\sigma\prec p$ to mean that $\sigma$ is a prefix of $p$, and define $\uparrow\sigma = \{p\in\baire\,|\, \sigma\prec p\}$. Similar notation will also apply to $2^{<\omega}$, the set of finite binary sequences.

We denote a topological space with underlying set $X$ and topology $\tau$ by $(X,\tau)$. If $\tau$ is clear from context, then we will often abbreviate $(X,\tau)$ by $X$. We will always assume that $\baire$ has the product topology, which is generated by sets of the form $\uparrow\sigma$ for $\sigma\in\omega^{<\omega}$. The \emph{specialization order} on a topological space $X$ is defined as $x\leq y$ if and only if $x$ is in the closure of $y$. A topological space $X$ is said to satisfy the \emph{$T_0$-separation axiom} if and only if $x\leq y$ and $y\leq x$ implies $x=y$ for all $x,y\in X$. A \emph{basis} for a topology $\tau$ is a family $\cal B\subseteq \tau$ such that every element of $\tau$ equals the union of elements of $\cal B$. A topological space is \emph{countably based} if and only if it has a basis with countably many elements. We emphasize that we always assume that a basis for a topology contains only open sets, which differs slightly from the definition used in \cite{etal_scott}.

Let $X$ be a set and $d$ a metric on $X$. A sequence $\{x_n\}_{n\in\omega}$ in $X$ is a \emph{Cauchy sequence} if and only if $\lim_{m,n} d(x_m,x_n)=0$. The metric space $(X,d)$ is \emph{complete} if and only if every Cauchy sequence has a limit in $X$. A topological space $X$ is \emph{completely metrizable} if and only if there is a metric $d$ compatible with the topology on $X$ such that $(X,d)$ is complete. A topological space is \emph{Polish} if and only if it is separable and completely metrizable.

We write $f\colon \subseteq X\to Y$ to denote that $f$ is a \emph{partial} function from $X$ to $Y$. The \emph{domain} of $f$ is the subset of $X$ for which $f$ is defined, and will be denoted $dom(f)$. A partial function $f\colon\subseteq X\to Y$ is continuous if and only if the preimage of every open subset of $Y$ is open with respect to the subspace topology on $dom(f)$.

Let $(P,\sqsubseteq)$ be a partially ordered set. A subset $D\subseteq P$ is \emph{directed} if and only if $D$ is non-empty and every pair of elements in $D$ has an upper bound in $D$. $P$ is a \emph{directed complete partially ordered set} (dcpo) if and only if every directed subset $D$ of $P$ has a supremum $\bigsqcup D$ in $P$. Given $x,y\in P$, $x$ is \emph{way below} $y$, written $x \ll y$, if and only if for every directed $D\subseteq P$ for which $\bigsqcup D$ exists, if $y\sqsubseteq  \bigsqcup D$ then there is $d\in D$ with $x\sqsubseteq d$. For $x\in P$ we define $\wayabovearrow x=\{y\in P\,|\, x \ll y\}$ and $\waybelowarrow x=\{y\in P\,|\, y\ll x\}$. An element $x\in P$ is \emph{compact} if and only if $x\ll x$.

A subset $U$ of $P$ is \emph{Scott-open} if and only if $U$ is an upper set (i.e., $x\in U$ and $x\sqsubseteq y$ implies $y\in U$) and for every directed $D\subseteq P$, if $\bigsqcup D$ exists and is in $U$ then $D\cap U\not=\emptyset$. The Scott-open subsets of $P$ form a topology on $P$ called the \emph{Scott-topology}.

A subset $B$ of $P$ is a (domain theoretic) \emph{basis} for $P$ if and only if for every $x\in P$, the set $B\cap \waybelowarrow x$ contains a directed subset with supremum equal to $x$. $P$ is an \emph{$\omega$-continuous domain} if and only if $P$ is a dcpo with a countable (domain theoretic) basis, and $P$ is an \emph{$\omega$-algebraic domain} if and only if $P$ is an $\omega$-continuous domain with a basis consisting only of compact elements. If $P$ is an $\omega$-continuous domain and $B$ is a countable (domain theoretic) basis for $P$, then $\wayabovearrow x$ is Scott-open for each $x\in P$ and $\{\wayabovearrow x\,|\, x\in B\}$ is a countable (topological) basis for the Scott-topology on $P$.

We let $\cal P(\omega)$ denote the power set of $\omega$ ordered by subset inclusion. $\cal P(\omega)$ is an $\omega$-algebraic domain, and the compact elements are precisely the finite subsets of $\omega$. We will always assume the Scott-topology on $\cal P(\omega)$, which is generated by sets of the form $\uparrow\!F = \{ X\in \cal P(\omega) \,|\, F\subseteq X\}$ with $F\subseteq \omega$ finite.


\section{Borel Hierarchy}\label{sec:borelhierarchy}

It is common for non-Hausdorff spaces to have open sets that are not $F_\sigma$ (i.e., countable unions of closed sets) and closed sets that are not $G_\delta$ (i.e., countable intersections of open sets). The Sierpsinski space, which has $\{\bot,\top\}$ as an underlying set and the singleton $\{\top\}$ open but not closed, is perhaps the simplest example of this phenomenon. This implies that the classical definition of the Borel hierarchy, which defines level $\bsigma 2$ as the $F_\sigma$-sets and $\bpi 2$ as the $G_\delta$-sets, is not appropriate in the general setting. We can overcome this problem by using the following modification of the Borel hierachy due to Victor Selivanov (see \cite{selivanov1984, selivanov2004,selivanov}).

\begin{definition}\label{def:borel_hierarchy}
Let $(X,\tau)$ be a topological space. For each ordinal $\alpha$ ($1\leq \alpha < \omega_1$) we define $\bsigma \alpha(X,\tau)$ inductively as follows.
\begin{enumerate}
\item
$\bsigma 1(X,\tau)=\tau$.
\item
For $\alpha>1$, $\bsigma \alpha(X,\tau)$ is the set of all subsets $A$ of $X$ which can be expressed in the form
\[A=\bigcup_{i\in\omega}B_i\setminus B'_i,\]
where for each $i$, $B_i$ and $B'_i$ are in $\bsigma{\beta_i}(X,\tau)$ for some $\beta_i<\alpha$.
\end{enumerate}
We define $\bpi \alpha(X,\tau)= \{X\setminus A\,|\, A\in \bsigma \alpha(X,\tau)\}$ and $\bdelta \alpha(X,\tau)=\bsigma \alpha(X,\tau)\cap \bpi \alpha(X,\tau)$. Finally, we define ${\mathbf B}(X,\tau)=\bigcup_{\alpha<\omega_1} \bsigma \alpha(X,\tau)$ to be the \emph{Borel} subsets of $(X,\tau)$.
\qed
\end{definition}

When the topology is clear from context, we will usually write $\bsigma \alpha(X)$ instead of $\bsigma \alpha(X,\tau)$.

The definition above is equivalent to the classical definition of the Borel hierarchy on metrizable spaces, but differs in general. V. Selivanov has investigated this hierarchy in a series of papers, with an emphasis on applications to $\omega$-continuous domains (see \cite{selivanov} for an overview of results). D. Scott \cite{scott_datatypes} and his student A. Tang \cite{tang1979, tang1981} have also investigated some aspects of the hierarchy in $\cal P(\omega)$, using the notation $\cal B_\sigma$ and $\cal B_\delta$ to refer to the levels $\bsigma 2$ and $\bpi 2$, respectively.

In the rest of this section, $X$ and $Y$ will denote arbitrary topological spaces, unless stated otherwise. The following results are easily proven, and can also be found in \cite{selivanov}.

\begin{proposition}
For each $\alpha$ ($1\leq \alpha < \omega_1$),
\begin{enumerate}
\item
$\bold{\Sigma}_{\alpha}^0(X)$ is closed under countable unions and finite intersections,
\item
$\bold{\Pi}_{\alpha}^0(X)$ is closed under countable intersections and finite unions,
\item
$\bold{\Delta}_{\alpha}^0(X)$ is closed under finite unions, finite intersections, and complementation.
\end{enumerate}
\qed
\end{proposition}

\begin{proposition}
If $\beta<\alpha$ then $\bold{\Sigma}_{\beta}^0(X)\cup \bold{\Pi}_{\beta}^0(X)\subseteq \bold{\Delta}_{\alpha}^0(X)$.
\qed
\end{proposition}

\begin{proposition}\label{prop:hierarchy_on_subspaces}
If $X$ is a subspace of $Y$, then $\bold{\Sigma}_{\alpha}^0(X)=\{ A\cap X\,|\, A\in \bold{\Sigma}_{\alpha}^0(Y) \}$ and $\bold{\Pi}_{\alpha}^0(X)=\{ A\cap X\,|\, A\in \bold{\Pi}_{\alpha}^0(Y) \}$.
\qed
\end{proposition}

\begin{proposition}
If $f\colon X\to Y$ is continuous, and $A\in \bsigma \alpha(Y)$ ($1\leq \alpha\leq \omega_1$), then $f^{-1}(A)\in \bsigma \alpha(X)$.
\qed
\end{proposition}

The next two propositions show that the Borel hierarchy in Definition \ref{def:borel_hierarchy} is equivalent to the classical definition for the case of metrizable spaces. This result is known (see \cite{selivanov}), but we include proofs for completeness.

\begin{proposition}\label{prop:borel3_and_up}
For $\alpha>2$, each $A\in\bold{\Sigma}_{\alpha}^0(X)$ can be expressed in the form
\[A=\bigcup_{i\in\omega}B_i,\]
where for each $i$, $B_i$ is in $\bold{\Pi}_{\beta_i}^0(X)$ for some $\beta_i<\alpha$.
\qed
\end{proposition}
\begin{proof}
Let $A=\bigcup_{i\in\omega} D_i\setminus D'_i$ with $D_i,D'_i \in \bsigma {\beta_i}(X)$ and $\beta_i<\alpha$. Since $\alpha>2$ we can assume $\beta_i\geq 2$ and write $D_i = \bigcup_{j\in\omega} G_{i,j}\setminus G'_{i,j}$ with $G_{i,j},G'_{i,j}\in \bsigma {\gamma_{i,j}}(X)$ and $\gamma_{i,j}<\beta_i$. Finally, $A=\bigcup_{i,j\in\omega} B_{i,j}$, where $B_{i,j} = G_{i,j}\setminus  (G'_{i,j}\cup D'_i)$ is in $\bpi {\beta_i}(X)$ because it is the intersection of the  $\bpi {\beta_i}$-sets $G_{i,j}$ and $X\setminus (G'_{i,j}\cup D'_i)$.
\qed
\end{proof}

\begin{proposition}
If $X$ is a metrizable space then every $A\in \bold{\Sigma}_2^0(X)$ is $F_\sigma$.
\end{proposition}
\begin{proof}
Assume $A=\bigcup_{i\in\omega} U_i\setminus V_i$, with $U_i$ and $V_i$ open. Every open subset of a metrizable space is $F_\sigma$ (see Proposition 3.7 in \cite{kechris}), so we can write $U_i = \bigcup_{j\in\omega} C_{i,j}$ with each $C_{i,j}$ closed. Then $A= \bigcup_{i,j\in\omega}\big( C_{i,j}\setminus V_i\big)$ is a countable union of closed sets.
\qed
\end{proof}

Singleton sets and diagonals of topological spaces are not closed in general, but the following results show that they are still well behaved for countably based $T_0$-spaces.

\begin{proposition}
If $X$ is a countably based $T_0$-space then every singleton set $\{x\}\subseteq X$ is in $\bpi 2(X)$.
\qed
\end{proposition}

For any topological space $X$ we let $\Delta_X = \{ \langle x,y \rangle\in X\times X \,|\, x = y \}$ be the diagonal of $X$.

\begin{proposition}
If $X$ is a countably based $T_0$-space then $\Delta_X \in \bold{\Pi}_2^0(X\times X)$.
\end{proposition}
\begin{proof}
Note that 
\[X\times X \setminus \Delta_X = \big(\bigcup_{i\in\omega}B_i\times (X\setminus B_i)\big)\cup \big(\bigcup_{i\in\omega}(X\setminus B_i)\times B_i \big),\]
where $\{B_i\}_{i\in\omega}$ is a countable basis for $X$. $B_i\times (X\setminus B_i)=(B_i\times X) \cap (X\times (X\setminus B_i))$, so it is the intersection of an open set and a closed set (similarly for $(X\setminus B_i)\times B_i$). It follows that $X\times X \setminus \Delta_X$ is in $\bsigma 2(X\times X)$, hence $\Delta_X$ is in $\bpi 2(X\times X)$.
\qed
\end{proof}

\begin{corollary}\label{cor:equalizer}
If $X$ and $Y$ are countably based $T_0$ spaces and $f,g\colon X\to Y$ are continuous functions, then $[f=g]:=\{ x\in X \,|\, f(x)=g(x)\}$ is in $\bpi 2(X)$.
\end{corollary}
\begin{proof}
Note that $\langle f, g\rangle \colon X \to Y\times Y$, defined as $x \mapsto \langle f(x),g(x)\rangle$, is continuous. Then $[f=g] = \langle f, g\rangle^{-1}(\Delta_Y)$ is in $\bpi 2(X)$.
\qed
\end{proof}

A special case of the above corollary was proven by D. Scott \cite{scott_datatypes} and A. Tang \cite{tang1979}.


\section{Quasi-metric spaces}\label{sec:quasimetricspaces}

Quasi-metrics are a generalization of metrics where the axiom of symmetry is dropped. These provide a useful way to generalize results from the theory of metric spaces to more general topological spaces.

\begin{definition}
A \emph{quasi-metric} on a set $X$ is a function $d\colon X\times X \to [0,\infty)$ such that for all $x,y,z\in X$:
\begin{enumerate}
\item
$x=y \iff d(x,y)=d(y,x)=0$
\item
$d(x,z)\leq d(x,y)+d(y,z)$.
\end{enumerate}
A \emph{quasi-metric space} is a pair $(X,d)$ where $d$ is a quasi-metric on $X$.
\qed
\end{definition}

Our general reference for quasi-metric spaces is \cite{kunzi_intro}. Note, however, that in \cite{kunzi_intro} the author reserves the term ``quasi-metric'' for spaces that satisfy the $T_1$-separation axiom, and our definition of quasi-metric is referred to as a ``$T_0$-quasi-pseudometric''.

A quasi-metric $d$ on $X$ induces a $T_0$ topology $\tau_d$ on $X$ generated by basic open balls of the form $B_d(x,\varepsilon) = \{ y\in X \,|\, d(x,y)<\varepsilon\}$ for $x\in X$ and real number $\varepsilon>0$.

We will sometimes call sets of the form $\overline{B}_d(x,\varepsilon) = \{ y\in X \,|\, d(x,y)\leq\varepsilon\}$ a closed ball, although in general they are not closed with respect to $\tau_d$.

If $(X,d)$ is a quasi-metric space, then $(X,\widehat{d} )$ is a metric space, where $\widehat{d}$ is defined as $\widehat{d}(x,y) = \max\{d(x,y), d(y,x)\}$. The metric topology induced by $\widehat{d}$ will be denoted $\tau_{\widehat{d}}$.

\begin{proposition}\label{prop:inducedmetric_opens_are_sigma2}
Let $(X,d)$ be a quasi-metric space. Every basic open ball $B_{\widehat{d}}(x,\varepsilon)$ is in $\bsigma 2(X,\tau_d)$.
\end{proposition}
\begin{proof}
Note that $B_{\widehat{d}}(x,\varepsilon)=\{y\in X\,|\, d(x,y)<\varepsilon \mbox{ \& } d(y,x)<\varepsilon\}=B_d(x,\varepsilon)\cap B_{d^{-1}}(x,\varepsilon)$, where we define $B_{d^{-1}}(x,\varepsilon) = \{y\in X\,|\, d(y,x)<\varepsilon\}$.

Let $R = \{r\in \bb Q\,|\, 0<r<\varepsilon\}$ and define $\overline{B}_{d^{-1}}(x,r) = \{y\in X\,|\, d(y,x)\leq r\}$ for $r\in R$. It is easy to see that $B_{d^{-1}}(x,\varepsilon) = \bigcup_{r\in R} \overline{B}_{d^{-1}}(x,r)$, so the proposition will be proved if we show (the well known fact) that $\overline{B}_{d^{-1}}(x,r)$ is $\tau_d$-closed for each $r \in R$.

Fix $r\in R$. For any $y\not\in \overline{B}_{d^{-1}}(x,r)$, there is some $r'>0$ such that $d(y,x)>r+r'$. If $z\in B_d(y,r')$ and $d(z,x)\leq r$ then $d(y,x)\leq d(y,z)+d(z,x)<r'+r$, a contradiction. Hence, $B_d(y,r')$ is a $\tau_d$-open neighborhood of $y$ which does not intersect $\overline{B}_{d^{-1}}(x,r)$, and it follows that $\overline{B}_{d^{-1}}(x,r)$ is closed with respect to $\tau_d$.
\qed
\end{proof}

\begin{proposition}[H.-P. A. K\"{u}nzi \cite{kunzi1983}]\label{prop:countablybased_iff_sep}
A quasi-metric space $(X,d)$ is countably based if and only if $(X,\widehat{d})$ is separable.
\end{proposition}
\begin{proof}
Let $\{U_i\}_{i\in\omega}$ be a countable basis for $(X,d)$. For $i,n\in\omega$, define
\[S_{\langle i,n\rangle}=\{x\in X\,|\,x\in U_i\subseteq B_d(x,2^{-n})\}.\]
If $S_{\langle i,n\rangle}$ is non-empty then choose $x_{\langle i,n\rangle}\in S_{\langle i,n\rangle}$. If $S_{\langle i,n\rangle}$ is empty, then let $x_{\langle i,n\rangle}$ be any element of $X$.

Now fix $n\in\omega$ and $x\in X$. Since $\{U_i\}_{i\in\omega}$ is a basis for $(X,d)$ there is $i\in\omega$ such that $x\in U_i \subseteq B_d(x,2^{-n})$. Thus $x\in S_{\langle i,n\rangle}$, so $S_{\langle i,n\rangle}$ is non-empty and $x_{\langle i,n\rangle}\in S_{\langle i,n\rangle}$. Then $d(x,x_{\langle i,n\rangle})<2^{-n}$ and $d(x_{\langle i,n\rangle},x)<2^{-n}$, so $\widehat{d}(x,x_{\langle i,n\rangle})<2^{-n}$. It follows that $\{x_{\langle i,n\rangle}\,|\, i,n\in\omega\}$ is a countable dense subset of $(X,\widehat{d})$.

For the converse, assume $D\subseteq X$ is countable dense with respect to $\tau_{\widehat{d}}$. Given $x\in X$ and $\varepsilon>0$, choose $y\in D$ and $n\in\omega$ so that $\widehat{d}(x,y)<2^{-n}<\varepsilon/2$. If $z\in B_d(y,2^{-n})$ then $d(x,z)\leq d(x,y)+d(y,z)<\varepsilon$, hence $z\in B_d(x,\varepsilon)$. Therefore, $x\in B_d(y,2^{-n})\subseteq B_d(x,\varepsilon)$, and it follows that $\{B_d(y,2^{-n})\,|\, \langle y,n\rangle\in D\times\omega\}$ is a countable basis for $\tau_d$.
\qed
\end{proof}

It follows from Propositions \ref{prop:inducedmetric_opens_are_sigma2} and \ref{prop:countablybased_iff_sep} that if $(X,d)$ is a countably based quasi-metric space, then every open subset of $(X,\widehat{d})$ is equal to a countable union of $\bsigma 2$-subsets of $(X,d)$. We therefore obtain the following.

\begin{theorem}\label{thrm:quasi_and_metric_borelrelation}
If $(X,d)$ is a countably based quasi-metric space, then the metric topology $\tau_{\widehat{d}}$ is a subset of $\bsigma 2(X,\tau_d)$. In particular, $\bsigma \alpha(X,\tau_d)= \bsigma \alpha(X,\tau_{\widehat{d}})$ for all $\alpha\geq\omega$, and ${\mathbf B}(X,\tau_d)={\mathbf B}(X,\tau_{\widehat{d}})$.
\qed
\end{theorem}

Recall that the Scott-topology on $\cal P(\omega)$ is generated by sets of the form $\uparrow\!F = \{X\in\cal P(\omega)\,|\, F\subseteq X \}$ with $F\subseteq \omega$ finite. A quasi-metric $d$ compatible with this topology on $\cal P(\omega)$ can be defined as 
\[d(X,Y) = \sup\{2^{-n}\,|\, n\in X\setminus Y\}\]
for $X,Y\subseteq \omega$, where we define the supremum of the empty set to be zero. In other words, $d(X,Y) = 2^{-n}$ where $n$ is the least element in $X$ and not in $Y$ if such an element exists, and $d(X,Y)=0$ if $X$ is a subset of $Y$. Then $\widehat{d}$ is the usual complete metric on $2^\omega$ if we identify elements of $\cal P(\omega)$ with their characteristic function. Selivanov \cite{selivanov} has shown that in this case $\bsigma n(\cal P(\omega),\tau_d)\not\subseteq \bpi n(\cal P(\omega),\tau_{\widehat{d}})$ and $\bpi n(\cal P(\omega),\tau_{\widehat{d}})\not\subseteq \bsigma {n+1}(\cal P(\omega),\tau_d)$ for all $n<\omega$.


\section{Complete quasi-metric spaces}\label{sec:complete_quasimetricspaces}

In the literature on quasi-metric spaces there are many competing definitions of ``Cauchy sequence'' and ``completeness''. The definition of ``Cauchy'' that we will adopt is sometimes called ``left K-Cauchy'' and our definition of completeness is sometimes called ``Smyth-complete'' (see \cite{kunzi_intro}). The main goal of this section is to characterize the countably based spaces which have topologies induced by a complete quasi-metric.

\begin{definition}[M. Smyth \cite{smyth1988}]
A sequence $(x_n)_{n\in\omega}$ in a quasi-metric space $(X,d)$ is \emph{Cauchy} if and only if for each real number $\varepsilon>0$ there exists $n_0\in\omega$ such that $d(x_n,x_m)<\varepsilon$ for all $m\geq n\geq n_0$. $(X,d)$ is a \emph{complete quasi-metric space} if and only if every Cauchy sequence in $X$ converges with respect to the metric topology $\tau_{\widehat{d}}$.
\qed
\end{definition}

To help build an intuition for the above definition, we consider two examples of quasi-metrics on the ordinal $\omega+1$. This set is naturally ordered as $0 < 1 < \cdots < \omega$.

The first quasi-metric we consider is defined as
\begin{eqnarray*}
d_1(x,y) &=& \left\{
\begin{array}{ll}
0 & \mbox{ if }x\leq y,\\
1/y & \mbox{ otherwise. }
\end{array}
\right.
\end{eqnarray*}
Verifying that $d_1$ is a quasi-metric can safely be left to the reader. The quasi-metric topology induced by $d_1$ is the \emph{Scott-topology}, which consists of the empty set and sets of the form $\uparrow\! x = \{y\in\omega+1\,|\, x\leq y\}$ for each $x\in\omega$. In this case, the singleton $\{\omega\}$ is \emph{not} open. The metric topology induced by $\widehat{d_1}$ has as open sets every $U\subseteq \omega+1$ such that either $\omega\not\in U$ or else $U\setminus (\omega+1)$ is finite. A sequence $(x_i)_{i\in\omega}$ is Cauchy with respect to $d_1$ if and only if for each $y\in\omega$ there is $i\in \omega$ such that $x_j\geq y$ for each $j\geq i$. Thus every Cauchy sequence is eventually constant or else converges to $\omega$ with respect to the metric topology $\tau_{\widehat{d_1}}$. This shows that $(\omega+1, d_1)$ is a complete quasi-metric space.

Next consider the quasi-metric defined as
\begin{eqnarray*}
d_2(x,y) &=& \left\{
\begin{array}{ll}
0 & \mbox{ if }x\leq y,\\
1 & \mbox{ otherwise. }
\end{array}
\right.
\end{eqnarray*}
The quasi-metric topology on $\omega+1$ induced by $d_2$ is the \emph{Alexandroff topology}, which consists of the empty set and sets of the form $\uparrow\! x = \{y\in\omega+1\,|\, x\leq y\}$ for each $x\in\omega+1$. In particular, the singleton $\{\omega\}$ is open in this topology. Note that the increasing sequence $\xi=(0,1,2,\ldots)$ is Cauchy with respect to $d_2$. However, the metric topology induced by $\widehat{d_2}$ is the discrete topology, so $\xi$ does not converge in $\tau_{\widehat{d_2}}$. This shows that $(\omega+1, d_2)$ is \emph{not} a complete quasi-metric space.

In both of the above examples, the natural order on $\omega+1$ can be completely recovered from the quasi-metrics by defining $x\leq y$ if and only if $d(x,y)=0$, and is also equivalent to the specialization order of the respective topologies (i.e., $x\leq y$ if and only if $x$ is in the closure of $\{y\}$). The Scott-topology induced by $d_1$ has the desirable property that increasing sequences like $0<1<2<\cdots$, which converge to $\omega$ in an \emph{order theoretical} sense, also converge to $\omega$ in the \emph{topological} sense. This is not true for the Alexandroff topology, because $\{\omega\}$ is open in that topology.

The notion of quasi-metric completeness defined above forces the order theoretical and topological aspects of quasi-metric spaces to be compatible in this sense. If we only required that a Cauchy sequence converge with respect to the quasi-metric topology, then $d_2$ would be included because every sequence converges to $0$ with respect to $\tau_{d_2}$. Furthermore, since $(\omega+1,\widehat{d_2})$ is a complete metric space, it is not an option to define completeness only in terms of the induced metric. We will see throughout this paper that the definition of completeness we have adopted is precisely what is needed to generalize the descriptive set theory of Polish spaces to the non-Hausdorff setting.

We will say that a topological space $(X,\tau)$ is \emph{completely quasi-metrizable} if and only if there is a complete quasi-metric $d$ on $X$ such that $\tau=\tau_d$.

\begin{definition}
A topological space is \emph{quasi-Polish} if and only if it is countably based and completely quasi-metrizable.
\qed
\end{definition}

If $(X,d)$ is a countably based complete quasi-metric space, then $(X,\widehat{d})$ is separable by Proposition \ref{prop:countablybased_iff_sep} and $\widehat{d}$ is complete because any sequence that is Cauchy with respect to $\widehat{d}$ is Cauchy with respect to $d$. Therefore, $(X,\widehat{d})$ has a Polish topology. We can immediately use this connection between quasi-Polish spaces and Polish spaces to make a few simple observations.

\begin{proposition}\label{prop:cantortheoremforquasipolish}
Every uncountable quasi-Polish space has cardinality $2^{\aleph_0}$.
\qed
\end{proposition}

We can also show that the fact that the Borel hierarchy on uncountable Polish spaces does not collapse (see, for example, Theorem 22.4 in \cite{kechris}) generalizes to uncountable quasi-Polish spaces.

\begin{theorem}\label{thrm:non_collapse}
If $X$ is an uncountable quasi-Polish space, then the Borel hierarchy on $X$ does not collapse.
\end{theorem}
\begin{proof}
Let $d$ be a compatible complete quasi-metric on $X$. Assume for a contradiction that $\bsigma \alpha(X,\tau_d)=\bpi \alpha(X,\tau_d)$ for some $\alpha<\omega_1$. Since $(X,\widehat{d})$ is an uncountable Polish space, there is $A\in \bsigma {\alpha+1}(X,\tau_{\widehat{d}})\setminus \bpi {\alpha+1}(X,\tau_{\widehat{d}})$. By Theorem~\ref{thrm:quasi_and_metric_borelrelation} we have $\tau_{\widehat{d}}\subseteq \bsigma 2(X,\tau_d)$. Therefore,
\[A\in \bsigma {\alpha+2}(X,\tau_d)=\bsigma {\alpha}(X,\tau_d)\subseteq \bsigma \alpha(X,\tau_{\widehat{d}})\subseteq \bpi {\alpha+1}(X,\tau_{\widehat{d}}),\]
a contradiction.
\qed
\end{proof}

V. Selivanov \cite{selivanov} has shown that the Borel hierarchy does not collapse for some uncountable $\omega$-continuous domains, including $\cal P(\omega)$. We will see later that every $\omega$-continuous domain is quasi-Polish, so the hierarchy does not collapse on any uncountable $\omega$-continuous domain.

Let $f\colon\subseteq X\to Y$ be a partial continuous function between quasi-metric spaces $(X,d)$ and $(Y,d')$. For $\varepsilon>0$, let $Q(f,\varepsilon)$ be the set of all $x\in X$ such that for every open neighborhood $U$ of $x$ there is open neighborhood $V\subseteq U$ of $x$ and $y\in f(V)$ such that $f(V)\subseteq B_{d'}(y,\varepsilon)$. Define $Q(f) = \bigcap_{n\in\omega} Q(f,2^{-n})$.

The following theorem does not assume that the spaces are countably based.

\begin{theorem}\label{thrm:continuous_extensions}
Let $(X,d)$ be a quasi-metric space, $(Y,d')$ a complete quasi-metric space, and $f\colon\subseteq X\to Y$ partial continuous. Then $dom(f)\subseteq Q(f)$ and $f$ extends to a continuous function $g\colon Q(f) \to Y$.
\end{theorem}
\begin{proof}
It is straight forward to check that $dom(f)\subseteq Q(f)$. For $x\in Q(f)$, there is a sequence $U_0\supseteq U_1\supseteq U_2\supseteq\cdots$ of open neighborhoods of $x$ such that $U_n\subseteq B_d(x,2^{-n})$ and $f(U_n)\subseteq B_{d'}(y_n,2^{-n})$ for some $y_n\in f(U_n)$. Then $d'(y_n,y_m)<2^{-n}$ for all $m\geq n$, thus $(y_n)_{n\in\omega}$ is Cauchy and converges with respect to $\widehat{d'}$ to some $y\in Y$. Define $g(x)=y$. 

Note that for any $\varepsilon>0$, there is $n\in\omega$ such that $1/n<\varepsilon/2$ and $d'(y,y_n)<\varepsilon/2$, hence $f(U_n)\subseteq B_{d'}(y,\varepsilon)$. Thus, if $V\subseteq U_n$ is an open neighborhood of $x$ and $y'\in f(V)$ is such that $f(V)\subseteq B_{d'}(y',\varepsilon)$, then $\widehat{d'}(y,y')<\varepsilon$. This implies that the definition of $g(x)$ is independent of our choice of $U_n$ and $y_n$.

Finally, we show that $g\colon Q(f)\to Y$ is continuous. If $x\in Q(f)$ and $\varepsilon>0$, then there is an open neighborhood $V$ of $x$ such that $f(V)\subseteq B_{d'}(g(x),\varepsilon/2)$. For any $x'\in V\cap Q(f)$, $g(x')$ is the limit (with respect to $\tau_{\widehat{d'}}$) of a Cauchy sequence of elements in $f(V)$, hence $g(x')\in B_{d'}(g(x),\varepsilon)$. Therefore, $V$ is an open neighborhood of $x$ satisfying $g(V)\subseteq B_{d'}(g(x),\varepsilon)$, and as $\varepsilon>0$ was arbitrary, $g$ is continuous at $x$.
\qed
\end{proof}

The above theorem should be compared with Theorem 3.8 in \cite{kechris} and Proposition II-3.9 in \cite{etal_scott}. The set $Q(f)$ defined above is closely related to the set of all $x$ such that $osc_f(x)=0$ as defined in \cite{kechris}. For arbitrary $f$ with metrizable codomain, the points of continuity of $f$ are precisely those $x$ satisfying $osc_f(x)=0$. However, this does not hold when we extend the definition of $Q(f)$ above to arbitrary functions. For example, take the Sierpinski space $X=\{\bot,\top\}$ with quasi-metric $d(\bot,\top)=0$ and $d(\top,\bot)=1$, and let $f$ be the function that swaps $\bot$ and $\top$. Then $Q(f)=X$, but $f$ is clearly not continuous.

The reader should note how completeness of the quasi-metric is used in the above proof. The theorem would not hold if we only required that $(Y,\widehat{d'})$ be a complete metric space (i.e., that $(Y,d)$ is bicomplete; see Section \ref{sec:othercompletegeneralmetrics}). For example, using the quasi-metrics on $\omega+1$ defined earlier, let $f\colon\subseteq (\omega+1, d_1)\to (\omega+1,d_2)$ be the restriction of the identity function on $\omega+1$ to the subspace $\{0,1,\ldots\}$. Then $f$ is a partial continuous function and $Q(f)=\omega+1$, but $f$ cannot be extended to a total continuous function because $\{\omega\}$ is open in $(\omega+1,d_2)$ but not in $(\omega+1,d_1)$.

\begin{theorem}\label{thrm:continuous_extensions_pi2domain}
If $X$, $Y$, and $f\colon\subseteq X\to Y$ are as in Theorem \ref{thrm:continuous_extensions}, and in addition $X$ is countably based, then $Q(f)\in \bpi 2(X)$.
\end{theorem}
\begin{proof}
Let $\{U_i\}_{i\in\omega}$ be a countable basis for $X$. For $i,n\in\omega$ let $A_{i,n}$ be the set of all $x\in U_i$ such that for every open neighborhood $V\subseteq U_i$ of $x$, $f(V)\not\subseteq B_{d'}(y,2^{-n})$ for every $y\in f(V)$. Clearly, if $x\not\in Q(f,2^{-n})$, then $x\in A_{i,n}$ for some $i\in\omega$. Furthermore, if we let $Cl(\cdot)$ be the closure operator on $X$, then $U_i\cap Cl(A_{i,n})\cap Q(f,2^{-n})=\emptyset$ because if $x\in U_i\cap Cl(A_{i,n})$ and $V\subseteq U_i$ is an open neighborhood of $x$ then $V$ is an open neighborhood of some element of $A_{i,n}$. Therefore, $X\setminus Q(f,2^{-n}) = \bigcup_{i\in\omega} U_i\cap Cl(A_{i,n})$ is in $\bsigma 2(X)$, hence $Q(f)$ is the countable intersection of $\bpi 2$ sets.
\qed
\end{proof}

\begin{theorem}
If $X$ is a countably based quasi-metric space and $Y\subseteq X$ is quasi-Polish, then $Y\in\bpi 2(X)$.
\end{theorem}
\begin{proof}
Let $f\colon Y\to Y$ be the identity on $Y$. Then $f\colon\subseteq X\to Y$ is partial continuous and extends to a continuous function $g\colon Q(f)\to Y$ with $Q(f)\in\bpi 2(X)$. Then $Y=\{x\in Q(f)\,|\, g(x)=x\}$, hence $Y\in \bpi 2(Q(f))$ by Corollary \ref{cor:equalizer}. It follows that $Y\in\bpi 2(X)$.
\qed
\end{proof}

We now prove the converse of the above theorem. The following does not require spaces to be countably based. See Theorem 3.11 in \cite{kechris} for the corresponding proof for completely metrizable spaces.

\begin{theorem}\label{thrm:bpi2_subspace_complete}
Every $\bpi 2$-subspace of a complete quasi-metric space is completely quasi-metrizable.
\end{theorem}
\begin{proof}
Assume $(X,d)$ is a complete quasi-metric space and $Y\in\bpi 2(X)$. Then $Y$ can be represented as $Y=\bigcap_{i\in\omega}(U_i \cup A_i)$, with $U_i\subseteq X$ open, $A_i\subseteq X$ closed, and $U_i\cap A_i = \emptyset$ for all $i\in\omega$. For $x,y\in Y$ and $i\in\omega$, define
\begin{eqnarray*}
d_i(x,y) &=& \left\{
\begin{array}{ll}
\min\{2^{-i-1},\max\{0,\frac{1}{d(y,F_i)}-\frac{1}{d(x,F_i)}\}\} & \mbox{ if }x,y\in U_i\\
2^{-i-1} & \mbox{ if }x\in U_i\mbox{ and } y\in A_i\\
0 & \mbox{ if } x\in A_i
\end{array}
\right.
\end{eqnarray*}
where $F_i = X\setminus U_i$ and $d(x,F_i) = \inf\{d(x,z) \,|\, z\in F_i\}$. Finally, define
\[d'(x,y) = d(x,y) + \sum_{i=0}^\infty d_i(x,y)\]
for $x,y\in Y$. It is straightforward to verify that $d'$ is a quasi-metric on $Y$. We show that $d'$ is compatible with the relative topology on $Y$ and that $(Y,d')$ is complete.

Clearly, for all $x\in Y$ and $\varepsilon>0$, $B_{d'}(x,\varepsilon)\subseteq B_d(x,\varepsilon)$, so the topology induced by $d'$ contains the subspace topology on $Y$. To prove that both topologies coincide, we must show that for all $x\in Y$ and $\varepsilon'>0$ there is $\varepsilon>0$ such that $B_d(x,\varepsilon)\subseteq B_{d'}(x,\varepsilon')$. So let $x\in Y$ and $\varepsilon'>0$ be fixed.

Choose $n>0$ large enough that $\sum_{i=n}^\infty 2^{-i-1}<\varepsilon'/3$. For each $i<n$, if $x\in A_i$ then let $\varepsilon_i=\varepsilon'/3$. Otherwise, $x\in U_i$ so we can choose $\varepsilon_i>0$ so that
\begin{enumerate}
\item
$\varepsilon_i<\varepsilon'/3$, 
\item
$B_d(x,\varepsilon_i)\subseteq U_i$,
\item
$y\in B_d(x,\varepsilon_i)$ implies $\frac{1}{d(y,F_i)}-\frac{1}{d(x,F_i)} < \varepsilon'/(3n)$.
\end{enumerate}
The second criterion can be met because $U_i$ is open. For the third criterion, note that if $d(x,y)<\varepsilon_i$ then $d(x,z) \leq d(x,y)+d(y,z) < \varepsilon_i + d(y,z)$ for all $z\in F_i$, hence $d(x,F_i) \leq d(y,F_i)+\varepsilon_i$. Therefore, if $\varepsilon_i < d(x,F_i)$ and $d(x,y)<\varepsilon_i$, then $\frac{1}{d(y,F_i)}-\frac{1}{d(x,F_i)} \leq \frac{1}{d(x,F_i)-\varepsilon_i}-\frac{1}{d(x,F_i)}$. Clearly, the right hand side of this last inequality becomes arbitrarily small as $\varepsilon_i$ approaches zero.

Now let $\varepsilon = \min\{\varepsilon_i \,|\, i<n\}$. By our choice of $\varepsilon$ we have that for all $y\in Y$, if $d(x,y)<\varepsilon$ then
\begin{eqnarray*}
d'(x,y) &=& d(x,y) + \sum_{i=0}^{n-1} d_i(x,y) + \sum_{i=n}^\infty d_i(x,y)\\
	&<& \varepsilon'/3 + \varepsilon'/3 + \varepsilon'/3\\
	&=& \varepsilon',
\end{eqnarray*}
hence $B_d(x,\varepsilon)\subseteq B_{d'}(x,\varepsilon')$. Therefore, $d'$ induces the relative topology on $Y$.

Finally, we show that $d'$ is complete. Let $(x_n)_{n\in\omega}$ be a Cauchy sequence in $(Y,d')$. Then $(x_n)_{n\in\omega}$ is Cauchy in $(X,d)$, so it converges to some $x\in X$ with respect to the metric topology induced by $\widehat{d}$.

For each $i\in\omega$, since $d_i(x_n,x_m)$ $(n\leq m)$ converges to zero, there is $n_i\in\omega$ such that either $x_n \in U_i$ for all $n\geq n_i$, or else $x_n \in A_i$ for all $n\geq n_i$. If $x_n \in A_i$ for all $n\geq n_i$ then $x\in A_i$ because $A_i$ is closed. On the other hand, if $x_n \in U_i$ for all $n\geq n_i$ then $d(x_n,F_i)$ must be bounded away from zero, thus $x\in U_i$. Therefore, $x\in Y$.

If $(x_n)_{n\in\omega}$ is eventually in $A_i$ then $x\in A_i$, and $d_i(x_n,x)=d_i(x,x_n)=0$ when $n$ is large enough. Otherwise, $(x_n)_{n\in\omega}$ is eventually in $U_i$ and $x\in U_i$. Since $\widehat{d}(x_n,x)$ converges to zero, it follows that $\abs{\frac{1}{d(x_n,F_i)}-\frac{1}{d(x,F_i)}}$ converges to zero, hence $d_i(x_n,x)$ and $d_i(x,x_n)$ converge to zero. As $d_i$ is bounded by $2^{-i-1}$, the infinite sums in the definitions of $d'(x_n,x)$ and $d'(x,x_n)$ converge to zero as $n$ goes to infinity. It follows that $(x_n)_{n\in\omega}$ converges to $x$ with respect to $\widehat{d'}$, and that $(Y,d')$ is complete.
\qed
\end{proof}

\begin{theorem}\label{thrm:quasipolish_bpi2_characterization}
A subspace of a quasi-Polish space is quasi-Polish if and only if it is $\bpi 2$.\qed
\end{theorem}

$\cal P(\omega)$ is complete with respect to the quasi-metric $d$ defined after Theorem \ref{thrm:quasi_and_metric_borelrelation}. Since every countably based $T_0$-space can be embedded into $\cal P(\omega)$, we obtain the following.

\begin{theorem}
A space is quasi-Polish if and only if it is homeomorphic to a $\bpi 2$-subset of $\cal P(\omega)$.
\qed
\end{theorem}

Using the above theorem we can give a very simple alternative proof that partial continuous functions into a quasi-Polish space can be extended to a $\bpi 2$-domain.

\begin{corollary}\label{cor:cont_extension_to_pi2_domain}
Let $X$ be an arbitrary topological space, $Y$ a quasi-Polish space, and $f\colon\subseteq X\to Y$ partial continuous. Then there is $G\in \bpi 2(X)$ with $dom(f)\subseteq G$ and a continuous extension $g\colon G\to Y$ of $f$.
\end{corollary}
\begin{proof}
We can assume without loss of generality that $Y\in \bpi 2(\cal P(\omega))$. By Proposition II-3.9 in \cite{etal_scott}, there is a continuous extension $f^* \colon X\to \cal P(\omega)$ of $f$. Let $G = (f^*)^{-1}(Y)$ and let $g$ be the restriction of $f^*$ to $G$. It is easy to check that $G$ and $g$ satisfy the claim of the corollary.
\qed
\end{proof}

Finally, we mention that the class of quasi-Polish spaces is closed under retracts. Recall that a topological space $X$ is a \emph{retract} of $Y$ if and only if there exist continuous functions $s\colon X\to Y$ and $r\colon Y\to X$ such that $r\circ s$ is the identity on $X$.

\begin{corollary}
Any retract of a quasi-Polish space is quasi-Polish.
\end{corollary}
\begin{proof}
Assume $Y$ is quasi-Polish and $s\colon X\to Y$ and $r\colon Y\to X$ are continuous such that $r\circ s$ is the identity on $X$. Then $s(X) =\{ y\in Y \,|\, s(r(y))=y\}\in \bpi 2(Y)$ by Corollary \ref{cor:equalizer}. Since $s(X)$ is homeomorphic to $X$, it follows that $X$ is quasi-Polish.
\qed
\end{proof}

Similar results concerning retracts can be found in \cite{scott_datatypes} and \cite{tang1979}. Retracts also play an important role in V. Selivanov's \cite{selivanov} development of descriptive set theory for domains.


\section{Other notions of complete generalized metrics}\label{sec:othercompletegeneralmetrics}

Although complete quasi-metrics have many similarities with complete metrics, it is important to note that a quasi-metric does not always have a well-defined ``completion''. For example, define a quasi-metric $d$ on $\omega$ by $d(x,y)=0$ if $x\leq y$ and $d(x,y)=1$ if $x>y$. The topology of $(\omega,d)$ is simply the Scott-topology on $\omega$ with the usual order. Now assume for a contradiction that $(Y,d')$ is a complete quasi-metric space such that $X\subseteq Y$ and $d'$ agrees with $d$ on $X$. The sequence $0,1,2,\ldots$ is Cauchy in $(X,d)$, hence in $(Y,d')$, so this sequence converges to some $y\in Y$ with respect to $\widehat{d'}$. So there is some $x_0 \in\omega$ such that $\widehat{d'}(x,y)<1/2$ for all $x\geq x_0$. Thus $d'(x_0+1,x_0)\leq d'(x_0+1,y) + d'(y,x_0)<1$, a contradiction.

Clearly every countably based quasi-metric space can be embedded topologically into a quasi-Polish space. It is unclear, however, whether or not quasi-Polish spaces can be characterized by some other notion of a complete generalized metric which does have a well-defined completion.

Before we further investigate the properties of quasi-Polish spaces, in this section we will briefly investigate two other notions of complete generalized metrics which do have well defined completions: bicomplete quasi-metric spaces and complete partial metrics. The bicompletion of a quasi-metric space is discussed in \cite{kunzi_intro} and the completion of a partial metric is given in \cite{oneil}. In general, bicomplete quasi-metrizability is strictly more general than complete quasi-metrizability, which is strictly more general than complete partial metrizability.

\subsection{Bicomplete quasi-metric spaces}\label{subsec:Bicomplete}

Note that it is possible for $(X,\widehat{d})$ to be a separable complete metric space without $(X,d)$ being quasi-Polish. For example, the quasi-metric $d$ on $\omega$ defined at the top of this section induces a non-sober topology on $\omega$, which we will later see implies that $(\omega,d)$ is not quasi-Polish. However, $(\omega,\widehat{d})$ is a discrete metric space, hence complete. 

\begin{definition}
A quasi-metric space $(X,d)$ is \emph{bicomplete} if and only if $(X,\widehat{d})$ is a complete metric space.
\qed
\end{definition}

H. Junnila and H.-P. A. K\"{u}nzi \cite{junnila_kunzi} have shown that a metrizable space has a compatible bicomplete quasi-metric if and only if it is a $\bpi 3$-subset of every metrizable space in which it is embedded. The goal of this section is to generalize this result to all countably based bicompletely quasi-metrizable spaces. Thus bicompleteness generalizes completeness by exactly one step in the Borel hierarchy.

We first need to extend Corollary \ref{cor:cont_extension_to_pi2_domain} to a more general class of functions.

\begin{definition}
A function $f\colon X\to Y$ is $\bsigma \alpha$-measurable ($1\leq \alpha <\omega_1$) if and only if $f^{-1}(U)\in \bsigma \alpha(X)$ for all open $U\subseteq Y$.
\qed
\end{definition}

In particular, a function is continuous if and only if it is $\bsigma 1$-measurable. It is easily seen that if $f\colon X\to Y$ is $\bsigma \alpha$-measurable and $A\in \bpi 2(Y)$, then $f^{-1}(A)\in \bpi {\alpha+1}(X)$. Properties of $\bsigma \alpha$-measurable functions between countably based $T_0$-spaces have been investigated in \cite{debrecht_etal_4}.

\begin{lemma}\label{lem:extensions_of_measurable_functions}
Let $X$ be an arbitrary topological space, $Y$ a quasi-Polish space, and $f\colon\subseteq X\to Y$ a partial $\bsigma \alpha$-measurable function ($1\leq \alpha <\omega_1$). Then there exists $G\in \bpi {\alpha+1}(X)$ with $dom(f)\subseteq G$ and a $\bsigma \alpha$-measurable extension $g\colon G\to Y$ of $f$.
\end{lemma}
\begin{proof}
Let $\{U_i\}_{i\in\omega}$ be a countable basis for $Y$, and for each $i\in\omega$ choose $V_i\in\bsigma \alpha(X)$ so that $V_i\cap dom(f)=f^{-1}(U_i)$. Let $\tau$ be the topology of $X$, and let $\tau'$ be the topology on $X$ generated by adding $\{V_i\,|\, i\in\omega\}$ to $\tau$. Then $f\colon\subseteq (X,\tau') \to Y$ is continuous, so there is $G\in\bpi 2(X,\tau')$ and a continuous extension $g\colon \subseteq (X,\tau')\to Y$ of $f$ with $dom(g)=G$. Since the identity function $id_X\colon (X,\tau)\to (X,\tau')$ is $\bsigma \alpha$-measurable, it is clear that $G\in\bpi {\alpha+1}(X,\tau)$ and $g\colon (X,\tau)\to Y$ is a $\bsigma \alpha$-measurable extension of $f$.
\qed
\end{proof}

The next theorem provides a useful tool for proving that a topological space admits a compatible bicomplete quasi-metric.

\begin{theorem}[S. Romaguera and O. Salbany \cite{romaguera_salbany}]\label{thrm:bicomplete_if_polishmetric}
A topological space $X$ admits a compatible bicomplete quasi-metric if and only if there is a compatible quasi-metric $d$ on $X$ such that $(X,\widehat{d})$ is completely metrizable.
\qed
\end{theorem}

As a simple consequence of the above theorem, note that if $(X,d)$ is a bicomplete quasi-metric space and $Y\in\bpi 2(X)$, then $(Y,\widehat{d})$ is a $\bpi 2$-subspace of $(X,\widehat{d})$ hence completely metrizable, thus it follows that $Y$ is bicompletely quasi-metrizable. 

\begin{lemma}\label{lem:sig2_bicomplete}
If $(X,d)$ is a bicomplete quasi-metric space and $Y\in\bsigma 2(X)$, then $Y$ admits a compatible bicomplete quasi-metric.
\end{lemma}
\begin{proof}
Let $Y=\bigcup_{i\in\omega} U_i\setminus V_i$ with $U_i,V_i$ open and $V_i\subseteq U_i$ for all $i\in\omega$. We first inductively define $A_\sigma,B_\sigma\subseteq X$ for $\sigma\in 2^{<\omega}$ as follows:
\begin{enumerate}
\item
$B_{\langle \rangle}=X$, where $\langle \rangle$ is the empty sequence,
\item
$B_{\sigma\diamond 0} = B_\sigma \setminus U_{\abs{\sigma}}$,
\item
$B_{\sigma\diamond 1} = B_\sigma \cap V_{\abs{\sigma}}$,
\item
$A_\sigma = B_\sigma\cap(U_{\abs{\sigma}}\setminus V_{\abs{\sigma}})$.
\end{enumerate}

Thus $A_{\langle\rangle} = U_0\setminus V_0$, $A_{\langle 0 \rangle}= (U_1\setminus V_1)\setminus U_0$, $A_{\langle 1 \rangle}= (U_1\setminus V_1)\cap V_0$, etc. In general, the sets $A_\sigma$ with $\abs{\sigma}=i$ form a partition of $(U_i\setminus V_i)\setminus \bigcup_{j<i} (U_j\setminus V_j)$, where the assumption that $V_j\subseteq U_j$ guarantees that the elements of the partition do not overlap. It is then easy to see that $Y=\bigcup_{\sigma\in 2^{<\omega}}A_\sigma$ and that $\sigma\not=\sigma'$ implies $A_\sigma\cap A_{\sigma'}=\emptyset$.

For $\sigma,\sigma'\in 2^{<\omega}$, let $\sigma\wedge\sigma'$ denote the longest common prefix of $\sigma$ and $\sigma'$. We define a total ordering $\sqsubseteq$ on the elements of $2^{<\omega}$ by
\[\sigma\sqsubseteq \sigma' \iff \sigma=\sigma'\mbox{ or } [(\sigma\wedge\sigma')\diamond 0 \preceq \sigma] \mbox{ or } [(\sigma\wedge\sigma')\diamond 1 \preceq \sigma'].\]
Intuitively, if we think of $2^{<\omega}$ as a binary tree with zeros branching to the left and ones branching to the right, and then collapse this tree vertically into a straight line, then $\sigma \sqsubseteq \sigma'$ if and only if $\sigma$ is to the left of $\sigma'$.

We now define a quasi-metric $\rho$ on $Y$. For $x,y\in Y$, let $\sigma_x$ ($\sigma_y$) be the unique element of $2^{<\omega}$ such that $x\in A_{\sigma_x}$ ($y\in A_{\sigma_y}$), and define
\begin{eqnarray*}
\rho(x,y) &=& \left\{
\begin{array}{ll}
d(x,y)+1 & \mbox{ if $\sigma_y \sqsubseteq \sigma_x$ and $\sigma_y\not=\sigma_x$,}\\
d(x,y)   & \mbox{ otherwise.}
\end{array}
\right.
\end{eqnarray*}
Since $\sqsubseteq$ is a total order, it is immediate that $\rho$ is a quasi-metric. 

We first show that the topology induced by $\rho$ is the same as the topology induced by $d$. Since $\rho(x,y)\leq d(x,y)$, it is clear that every $d$-open set is $\rho$-open. For the converse, let $V$ be $\rho$-open and choose $x\in V$ and $\varepsilon>0$ such that $B_\rho(x,\varepsilon)\subseteq V$. Define
\[U = U_{\abs{\sigma_x}}\cap\bigcap\{V_i\,|\, \sigma_x(i)=1\}.\]
Clearly $U$ is $d$-open, so choose $0<r<1$ such that $r<\varepsilon$ and $B_d(x,r)\subseteq U$. Now let $y\in Y$ be any point such that $d(x,y)<r$, and assume for a contradiction that $\rho(x,y)>r$. Then $\sigma_y\sqsubseteq \sigma_x$ and $\sigma_y \not=\sigma_x$. 

If $(\sigma_y\wedge\sigma_x)\diamond 0 \preceq \sigma_y$ then $A_{\sigma_y}\cap U_{\abs{\sigma_y\wedge\sigma_x}}$ is empty, hence $y\not\in U_{\abs{\sigma_y\wedge\sigma_x}}$. Since by assumption $y\in U \subseteq U_{\abs{\sigma_x}}$, it follows that $\sigma_x\not=(\sigma_y\wedge\sigma_x)$ hence $(\sigma_y\wedge\sigma_x)\diamond 1 \preceq \sigma_x$ because $\sigma_y\wedge\sigma_x$ is the longest common prefix of $\sigma_x$ and $\sigma_y$. But then $\sigma_x(\abs{\sigma_y\wedge\sigma_x})=1$, so $U\subseteq V_{\abs{\sigma_y\wedge\sigma_x}} \subseteq U_{\abs{\sigma_y\wedge\sigma_x}}$, contradicting $y\in U$.

On the other hand, if $(\sigma_y\wedge\sigma_x)\diamond 1 \preceq \sigma_x$, then $\sigma_x(\abs{\sigma_y\wedge\sigma_x})=1$, so $U\subseteq V_{\abs{\sigma_y\wedge\sigma_x}}$. Thus $y\in V_{\abs{\sigma_y\wedge\sigma_x}}$ so we must have $\sigma_y = \sigma_y\wedge\sigma_x$. But by definition, $A_{\sigma_y} \cap V_{\sigma_y}=\emptyset$, contradiction.

It follows that $B_d(x,r)\subseteq B_\rho(x,\varepsilon)\subseteq V$, hence $V$ is $d$-open.

Finally, each $A_\sigma$ is $\bpi 2$ in $(X,\widehat{d})$, hence $(A_\sigma,\widehat{d})$ is Polish. Since $\rho$ coincides with $d$ on $A_\sigma$, it follows that $(A_\sigma,\widehat{\rho})$ is Polish. For $x,y\in Y$, if $\sigma_x\not=\sigma_y$ then $\widehat{\rho}(x,y)\geq 1$, so each $A_\sigma$ is clopen in $(Y,\widehat{\rho})$. It follows that $(Y,\widehat{\rho})$ is the disjoint union of countably many Polish spaces, hence Polish. It follows from Theorem \ref{thrm:bicomplete_if_polishmetric} that $Y$ admits a compatible bicomplete quasi-metric.
\qed
\end{proof}

The general strategy in the above proof is essentially the same as the proof in \cite{junnila_kunzi}, but in order to guarantee that $\rho$ is a compatible quasi-metric we required a more complicated partitioning of $Y$.

We can now prove the following characterization of countably based spaces which admit bicomplete quasi-metrics.

\begin{theorem}\label{thrm:bicomplete_characterization}
Assume $X$ is a countably based space which admits a compatible bicomplete quasi-metric. Then $Y\subseteq X$ admits a compatible bicomplete quasi-metric if and only if $Y\in\bpi 3(X)$.
\end{theorem}
\begin{proof}
Assume $d$ is a bicomplete quasi-metric compatible with the subspace topology of $Y$. Let $f\colon (Y,d)\to (Y,\widehat{d})$ be the identity function on $Y$. Theorem \ref{thrm:quasi_and_metric_borelrelation} implies that $f$ is $\bsigma 2$-measurable, hence by Lemma \ref{lem:extensions_of_measurable_functions} there is a $\bsigma 2$-measurable extension $g\colon \subseteq X\to (Y,\widehat{d})$ of $f$ with $dom(g)\in\bpi 3(X)$. Clearly, $f^{-1}\colon (Y,\widehat{d})\to (Y,d)$ is a continuous function, hence $(f^{-1})\circ g \colon dom(g)\to (Y,d)$ is $\bsigma 2$-measurable. Using a simple generalization of the proof of Corollary \ref{cor:equalizer} we see that $Y = \{x\in dom(g)\,|\, x = f^{-1}(g(x))\}\in \bpi 3(dom(g))$. Since $dom(g)\in\bpi 3(X)$, it follows that $Y\in \bpi 3(X)$.

The proof of the converse is a simple generalization of the proof given in \cite{junnila_kunzi}. Assume $Y\in \bpi 3(X)$. Using Proposition \ref{prop:borel3_and_up} we have $Y=\bigcap_{i\in\omega} A_i$ for some choice of sets $A_i\in\bsigma 2(X)$. By Lemma \ref{lem:sig2_bicomplete} each $A_i$ admits a bicomplete quasi-metric $d_i$. The topological product $\prod_{i\in\omega}A_i$ admits a bicomplete quasi-metric $d$ defined as
\[d(x,y) = \sum_{i\in\omega}2^{-i}\frac{d_i(x_i,y_i)}{1+d_i(x_i,y_i)}\]
for $x=(x_i)_{i\in\omega}$ and $y=(y_i)_{i\in\omega}$. Note that $Y$ is homeomorphic to the subspace $D=\{(x,x,\ldots)\,|\, x\in Y\}$ of $\prod_{i\in\omega}A_i$. Let $\{B_k\}_{k\in\omega}$ be a countable basis for $X$. Then $(x_i)_{i\in\omega}\in D$ if and only if $(\forall i,j,k)[x_i\in A_i\cap B_k \iff x_j\in A_j\cap B_k]$, which implies that $D\in \bpi 2(\prod_{i\in\omega}A_i)$. Therefore, $Y$ admits a bicomplete quasi-metric.
\qed
\end{proof}

\begin{theorem}
A countably based $T_0$-space has a compatible bicomplete quasi-metric if and only if it is homeomorphic to a $\bpi 3$-subset of $\cal P(\omega)$.
\qed
\end{theorem}

Thus completeness corresponds to level $\bpi 2$ of the Borel hierarchy, bicompleteness corresponds to level $\bpi 3$, and bicompleteness is defined in terms of a quasi-metric which induces a $\bpi 2$ metric topology. H.-P. A. K\"{u}nzi and E. Wajch \cite{kunzi1997} have shown that this pattern extends to higher levels of the Borel hierarchy for metrizable quasi-metric spaces, and we expect that their results will generalize to all countably based quasi-metric spaces.

\subsection{Complete partial metric spaces}

Partial metrics are another generalization of metrics to non-Hausdorff spaces. Unlike quasi-metrics, partial metrics are always symmetric, but they allow the distance from a point to itself (sometimes called the ``weight'' or ``self-distance'' of the point) to be non-zero. In many applications, the weight of a point is used to provide a quantification of the degree of ``uncertainty'' or ``incompleteness'' of the point.

\begin{definition}[S. Matthews \cite{Matthews}]
A \emph{partial metric} on a set $X$ is a function $p : X\times X \to [0,\infty)$ such that for all $x,y,z\in X$:
\begin{enumerate}
\item
$x=y \iff p(x,x)=p(x,y)=p(y,y)$
\item
$p(x,x)\leq p(x,y)$
\item
$p(x,y)=p(y,x)$
\item
$p(x,z)\leq p(x,y)+p(y,z)-p(y,y)$.
\end{enumerate}
A \emph{partial metric space} is a pair $(X,p)$ where $p$ is a partial metric on $X$.
\qed
\end{definition}

Given a partial metric space $(X,p)$, $x\in X$, and $\varepsilon>0$, the open ball $B_p(x,\varepsilon)$ centered at $x$ with radius $\varepsilon$ is defined as 
\[B_p(x,\varepsilon)=\{y\in X \,|\, p(x,y) - p(x,x) < \varepsilon\}.\]
The collection $\{B_p(x,\varepsilon) \,|\, x\in X, \varepsilon > 0\}$ of open balls is a basis for a $T_0$ topology on $X$, which we denote by $\tau_p$.

A sequence $(x_i)_{i\in\omega}$ of elements of a partial metric space $(X,p)$ is a \emph{Cauchy sequence} if and only if $\lim_{i,j\to\infty} p(x_i,x_j)$ exists. $(X,p)$ is a \emph{complete partial metric space} if and only if every Cauchy sequence $(x_i)_{i\in\omega}$ in $X$ converges (with respect to $\tau_p$) to an element $x\in X$ satisfying $p(x,x) = \lim_{i,j\to\infty} p(x_i,x_j)$.

\begin{proposition}
Every countably based complete partial metric space is quasi-Polish.
\end{proposition}
\begin{proof}
This follows from known results about the connections between partial metrics and quasi-metrics (see, for example, Theorem B in \cite{Romaguera} and the cited references therein). We include a proof for completeness.

Let $(X, p)$ be a complete partial metric space. Define $d_p(x,y) = p(x,y)-p(x,x)$ for each $x,y\in X$. The reader can verify that $d_p$ is a quasi-metric on $X$ that induces the same topology as $p$. We will show that $d_p$ is also complete. Fix a sequence $(x_i)_{i\in\omega}$ in $X$ which is Cauchy with respect to $d_p$. 

We first show that $(x_i)_{i\in\omega}$ is Cauchy with respect to $p$. Note that for any $i,j\in\omega$ and $\ell\in\bb R$, $\abs{p(x_i,x_j) - \ell} \leq \abs{p(x_i,x_j) - p(x_i,x_i)} + \abs{p(x_i,x_i) - \ell} = d_p(x_i,x_j) + \abs{p(x_i,x_i) - \ell}$. By the symmetry of $p$ and the assumption that  $(x_i)_{i\in\omega}$ is Cauchy with respect to $d_p$, we can conclude that $\lim_{i,j\to\infty} p(x_i,x_j)$ exists if and only if $(p(x_i,x_i))_{i\in\omega}$ converges. So let $\varepsilon>0$ be given and choose $m\in\omega$ large enough that  $d_p(x_i, x_j)<\varepsilon/2$ whenever $j \geq i \geq m$. Let $k = \inf_{i\geq m} p(x_i,x_i)$ and fix $n\geq m$ so that $p(x_n,x_n) < k + \varepsilon/2$. For any $i\geq n$, $d_p(x_n,x_i) = p(x_n,x_i) - p(x_n,x_n) < \varepsilon/2$, hence $k \leq p(x_i,x_i) \leq p(x_n,x_i) < p(x_n,x_n)+\varepsilon/2 < k + \varepsilon$. This implies that $(p(x_i,x_i))_{i\in\omega}$ is a Cauchy sequence in $\bb R$, hence it converges. It follows that $(x_i)_{i\in\omega}$ is Cauchy with respect to $p$.

Since $p$ is complete, $(x_i)_{i\in\omega}$ converges with respect to $\tau_p$ to some $x\in X$ satisfying $p(x,x)=\lim_{i,j\to\infty} p(x_i,x_j)$. Since $(x_i)_{i\in\omega}$ converges to $x$ with respect to $\tau_p$ it immediately follows that $d_p(x,x_i)$ converges to zero. Furthermore, $d_p(x_i,x) = p(x_i,x)-p(x_i,x_i) + p(x,x) - p(x,x) = d_p(x,x_i) + p(x,x) - p(x_i,x_i)$, hence $d_p(x_i,x)$ also converges to zero because $p(x_i,x_i)$ converges to $p(x,x)$. Therefore, $(x_i)_{i\in\omega}$ converges to $x$ with respect to $\widehat{d_p}$. This shows that $(X,d_p)$ is a complete quasi-metric space.
\qed
\end{proof}

$\cal P(\omega)$ can be equipped with a complete partial metric (see \cite{oneil}), hence every countably based $T_0$-space can be equipped with a topologically compatible partial metric. However, it is not true that every quasi-Polish space can be equipped with a compatible \emph{complete} partial metric.

\begin{proposition}
There exist quasi-Polish spaces which are not completely partially metrizable.
\end{proposition}
\begin{proof}
Consider the set $X = \omega \cup \{\bot_1,\bot_2\}$ with the partial ordering $\sqsubseteq$ defined so that $\bot_1$ and $\bot_2$ are incomparable, $\bot_1,\bot_2 \sqsubseteq n$ for all $n\in\omega$, and $n \sqsubseteq m \iff n\geq m$ for all $n,m\in\omega$ (so $X$ is an infinitely descending sequence with two incomparable ``bottom'' elements). Then $X$ with the Scott-topology is a quasi-Polish space with compatible complete quasi-metric $d(x,y) = 0$ if $x\sqsubseteq y$ and $d(x,y)=1$, otherwise. 

To see that $X$ is not completely partially metrizable, assume that $p$ is a partial metric on $X$. Since the specialization order on $X$ coincides with $\sqsubseteq$ we must have that $p(n,n) = p(n,m) \geq p(m,m)$ whenever $n\sqsubseteq m$. This implies that for all $n\in\omega$, $p(\bot_1,\bot_2)\leq p(\bot_1,n)+p(n,\bot_2)-p(n,n) = p(\bot_1,\bot_1)+p(\bot_2,\bot_2)-p(n,n)$.

Now consider the infinite descending sequence $0\sqsupseteq 1\sqsupseteq 2\sqsupseteq \cdots$. Then $p(0,0)<p(1,1)<p(2,2)<\cdots < p(\bot_1,\bot_1) <\infty$, so $(p(n,n))_{n\in\omega}$ converges, hence $\ell = \lim_{n,m\to\infty}p(n,m)$ exists (and is finite). Clearly, $\ell \not=p(n,n)$ for any $n\in\omega$. If $\ell = p(\bot_1,\bot_1)$, then since $p(\bot_1,\bot_2)\leq p(\bot_1,\bot_1)+p(\bot_2,\bot_2)-p(n,n)$ for all $n\in\omega$, it would follow that $p(\bot_1,\bot_2)\leq p(\bot_2,\bot_2)$. But this would mean that $\bot_1$ is in every open set containing $\bot_2$, which is a contradiction since $\bot_1$ and $\bot_2$ are incomparable under $\sqsubseteq$. Similarly, $\ell \not= p(\bot_2,\bot_2)$. Therefore, $(X,p)$ is not a complete partial metric space.
\qed
\end{proof}

Currently we do not know of a nice characterization of the quasi-Polish spaces which are completely partially metrizable.


\section{Open continuous surjections from quasi-Polish spaces}\label{sec:opencontsurjections}

In this section we characterize quasi-Polish spaces as precisely the images of Polish spaces under continuous open functions.

Recall that a function is \emph{open} if and only if the image of every open set is open. A closed set is \emph{irreducible} if and only if it is non-empty and not the union of two proper closed subsets. A space is \emph{sober} if and only if every irreducible closed set equals the closure of a unique point.

\begin{lemma}\label{lem:open_surj_sober}
If $X$ is a Polish space, $Y$ is a $T_0$-space, and $f\colon X\to Y$ is an open continuous surjection, then $Y$ is sober.
\end{lemma}
\begin{proof}
Let $C\subseteq Y$ be an irreducible closed set. Since $Y$ is a $T_0$-space, it suffices to show that $C$ is the closure of some $y\in Y$. 

Note that if $U,V\subseteq Y$ are open and have non-empty intersection with $C$, then $U\cap V\cap C$ is non-empty. Indeed, if $U\cap V\cap C=\emptyset$, then $C$ equals the union of the two closed sets $C\setminus U$ and $C\setminus V$, hence $U\cap C=\emptyset$ or $V\cap C=\emptyset$ by the irreducibility of $C$.

Fix a compatible complete metric on $X$ and let $\cal B$ be a countable basis for $Y$ (which exists because $X$ is separable and $f$ is an open continuous surjection). Let $(B_n)_{n\in\omega}$ be an enumeration (with possible repetitions) of all $B\in \cal B$ such that $B \cap C$ is non-empty.

Set $U_0=X$. For $n\geq 0$, $f(U_n)$ and $B_n$ are open and have non-empty intersection with $C$, hence $f(U_n)\cap B_n \cap C$ is non-empty. Choose some $x_n\in U_n \cap f^{-1}(B_n\cap C)$. Let $U_{n+1}$ be an open neighborhood of $x_n$ with diameter less than $1/n$ and closure contained in $U_n\cap f^{-1}(B_n)$.

Since $(x_n)_{n\in\omega}$ is Cauchy and $X$ is complete, $(x_n)_{n\in\omega}$ converges to some $x\in X$. Since $x_n \in f^{-1}(C)$ for each $n\in\omega$ and $f^{-1}(C)$ is closed, it follows that $x\in f^{-1}(C)$. Furthermore, for each $n\in\omega$, the closure of $U_{n+1}$ is a subset of $f^{-1}(B_n)$ and $x_m\in U_{n+1}$ for all $m>n$, hence $x\in f^{-1}(B_n)$.

Setting $y=f(x)$, we have $y\in C$ and $y\in B_n$ for all $n\in\omega$. Since every open subset of $Y$ that intersects $C$ contains $y$, it follows that $C$ is the closure of $y$.
\qed
\end{proof}

\begin{lemma}\label{lem:open_surj_from_baire}
If $X$ is non-empty and quasi-Polish then there exists an open continuous surjection $f\colon \baire\to X$.
\end{lemma}
\begin{proof}
Let $\delta\colon \baire\to \cal P(\omega)$ be defined as $p \mapsto \{ n \,|\, \exists m: p(m)=n+1 \}$. It is easy to see that $\delta$ is an open continuous surjection. If $X\in \bpi 2(\cal P(\omega))$, then $Y=\delta^{-1}(X)$ is Polish, and $\delta|_Y$, the restriction of $\delta$ to $Y$, is an open continuous surjection onto $X$. There is an open continuous surjection $h$ from $\baire$ to $Y$ (see Exercise 7.14 in \cite{kechris}), hence $\delta|_Y \circ h$ is an open continuous surjection from $\baire$ to $X$.
\qed
\end{proof}

\begin{corollary}
Every quasi-Polish space is sober.
\qed
\end{corollary}

As shown at the top of Section \ref{subsec:Bicomplete}, bicomplete quasi-metric spaces can fail to be sober.

\begin{theorem}\label{thrm:open_surjection_preserves_pi2}
If $X$ is quasi-Polish, $Y$ is a $T_0$-space, and $f\colon X\to Y$ is an open continuous surjection, then $Y$ is quasi-Polish.
\end{theorem}
\begin{proof}
By Lemma \ref{lem:open_surj_from_baire}, it suffices to prove the theorem for open continuous surjections $f\colon \baire\to Y$. Let $(\sigma_n)_{n\in\omega}$ be an enumeration of all finite sequences of natural numbers and let $B_n = \{ p\in \baire \,|\, \sigma_n \prec p\}$. Note that $\{f(B_n)\}_{n\in\omega}$ is a countable basis for the topology on $Y$. Let $\cal F$ be the set of all $F\subseteq \omega$ such that:
\begin{enumerate}
\item
$F\not=\emptyset$,
\item
$m\in F$ implies $(\exists n\in F)$ such that $\sigma_m \prec \sigma_n$,
\item
$m\in F$ and $f(B_m)\subseteq f(B_n)$ implies $n \in F$,
\item
$m,n\in F$ implies $(\exists k\in F)$ such that $f(B_k)\subseteq f(B_m)\cap f(B_n)$.
\end{enumerate}
Note that the third condition implies that if $m\in F$ and $\sigma_n\preceq \sigma_m$, then $n\in F$.

Define $\phi\colon Y\to \cal P(\omega)$ by $\phi(y)=\{ n\in\omega \,|\, y\in f(B_n)\}$. It is easy to see that $\phi$ is a topological embedding of $Y$ into $\cal P(\omega)$. We will show that $\phi(Y)=\cal F$, which will imply that $Y$ and $\cal F$ are homeomorphic. It is straight forward to check that $\phi(y)\in \cal F$ for each $y\in Y$, so it remains to show that each $F\in \cal F$ is equal to $\phi(y)$ for some $y\in Y$.

Let $F\in\cal F$ be given. Define $X_F=\{p\in \baire\,|\, (\forall \sigma_n\prec p) n\in F\}$. We will show that $f(X_F)$ is an irreducible closed subset of $Y$.

Note that for any $n_0\in F$ there exists $p\in B_{n_0}\cap X_F$. This is because the second condition on $F$ implies there is an infinite strictly ascending sequence $\sigma_{n_0}\prec \sigma_{n_1}\prec \sigma_{n_2} \ldots$ with $n_i\in F$ for all $i\in\omega$. Then $\bigcap_{i\in\omega} B_{n_i} = \{p\}$ for some $p\in\baire$. Clearly, if $\sigma_m\prec p$ then $\sigma_m\preceq \sigma_{n_i}$ for some $i\in\omega$, so the third condition on $F$ guarantees that $p\in X_F$. In particular, $X_F$ is non-empty because $F$ is non-empty.

We now verify that $f(X_F)$ is closed. Since $F$ encodes a pruned tree \cite{kechris}, it is easy to see that $X_F$ is a closed subset of $\baire$. It follows that $f(\baire\setminus X_F)$ is open because $f$ is an open map, so by showing $Y\setminus f(X_F) = f(\baire\setminus X_F)$ we can conclude that $f(X_F)$ is closed. The surjectivity of $f$ implies that $Y\setminus f(X_F)\subseteq f(\baire\setminus X_F)$. To prove $f(\baire\setminus X_F) \subseteq Y\setminus f(X_F)$, it suffices to show that for all $p\in \baire$, if $f(p)\in f(X_F)$ then $p\in X_F$. So assume $f(p)=f(q)$ for some $q\in X_F$, and choose any $n\in\omega$ such that $\sigma_n\prec p$. Then $f(B_n)$ is an open neighborhood of $f(q)$, so using the fact that $f$ is continuous and $q\in X_F$ there is $m\in F$ such that $\sigma_m \prec q$ and $q\in f(B_m) \subseteq f(B_n)$. By the third condition on $F$ it follows that $n\in F$, and since $n$ was arbitrary, $p\in X_F$.

Next we show that $f(X_F)$ is irreducible. Let $C_1$ and $C_2$ be two closed proper subsets of $f(X_F)$, and choose $y_1 \in f(X_F)\setminus C_1$ and $y_2 \in f(X_F)\setminus C_2$. Then there are $n_1,n_2\in \omega$ such that $y_i\in f(B_{n_i})$ and $f(B_{n_i})\cap C_i = \emptyset$ ($i\in\{1,2\}$). Since $f(p)\in f(X_F)$ implies $p\in X_F$, it follows that $n_1$ and $n_2$ are in $F$. By the fourth condition on $F$ there is $k\in F$ such that $f(B_k)\subseteq f(B_m)\cap f(B_n)$, which implies there is $p\in B_k\cap X_F$ with $f(p)\in f(B_m)\cap f(B_n)$. Clearly, $f(p)\in f(X_F)$ but $f(p)\not\in C_1\cup C_2$.

Since Lemma \ref{lem:open_surj_sober} implies that $Y$ is sober, $f(X_F)$ is the closure of a unique $y\in Y$. For any $n\in F$, $B_n\cap X_F$ is non-empty, hence $f(B_n)$ is an open set intersecting the closure of $\{y\}$, thus $y\in f(B_n)$. On the other hand, if $y\in f(B_n)$ then $y=f(p)$ for some $p\in B_n$, which implies $p\in B_n\cap X_F$, hence $n\in F$. Therefore, $F=\phi(y)$.

It follows that $Y$ is homeomorphic to $\cal F$. To complete the proof of the theorem, we only need to show that $\cal F\in \bpi 2(\cal P(\omega))$. For $m\in\omega$, define
\[U_m = \{ S \in\cal P(\omega) \,|\, m\in S\} \qquad\mbox{and}\qquad N_m = \{ S \in\cal P(\omega) \,|\, m\not\in S\}.\]
Note that $U_m$ is open and $N_m$ is closed in $\cal P(\omega)$.

\begin{enumerate}
\item
Define $\cal F_1 = \cal P(\omega)\setminus \{\emptyset\}$.
\item
For $m\in\omega$ let $I_m = \{n\in\omega \,|\, \sigma_m \prec \sigma_n\}$ and define 
\[\cal F_2 = \bigcap_{m\in\omega}\left(N_m\cup \bigcup_{n\in I_m}U_n\right).\]
\item
For $m\in\omega$ let $J_m = \{n\in\omega \,|\, f(B_m)\subseteq f(B_n)\}$ and define
\[\cal F_3 = \bigcap_{m\in\omega}\left(N_m\cup \bigcap_{n\in J_m}U_n\right).\]
\item
For $m,n\in\omega$ let $K_{m,n} = \{k\in\omega\,|\, f(B_k)\subseteq f(B_m)\cap f(B_n)\}$ and define
\[\cal F_4 = \bigcap_{m,n\in\omega}\left(N_m\cup N_n \cup \bigcup_{k\in K_{m,n}} U_k\right).\]
\end{enumerate}
Since $\bpi 2(\cal P(\omega))$ is closed under countable intersections and finite unions, it is easy to see that $\cal F_1, \cal F_2,\cal F_3$, and $\cal F_4$ are all in $\bpi 2(\cal P(\omega))$. It is also straight forward to check that $\cal F = \cal F_1\cap\cal F_2\cap\cal F_3\cap\cal F_4$, hence $\cal F$ is in $\bpi 2(\cal P(\omega))$. Therefore, $Y$ is quasi-Polish.
\qed
\end{proof}

\begin{theorem}
A non-empty $T_0$-space $X$ is quasi-Polish if and only if there exists a continuous open surjection from $\baire$ to $X$.
\qed
\end{theorem}

It is well known (see, for example, Theorem 8.19 in \cite{kechris}) that if $X$ is Polish, $Y$ is a separable metrizable space, and $f\colon X\to Y$ is a continuous open surjection, then $Y$ is Polish.

\begin{corollary}
A metrizable space is quasi-Polish if and only if it is Polish.
\qed
\end{corollary}

The next corollary follows by taking products (or disjoint unions) of suitable continuous open surjections.

\begin{corollary}\label{cor:countable_products}
Every countable product of quasi-Polish spaces is quasi-Polish, and every countable disjoint union of quasi-Polish spaces is quasi-Polish.
\qed
\end{corollary}


\section{Countably based locally compact sober spaces}\label{sec:locallycompactsober}

In this section we show that every countably based locally compact sober space is quasi-Polish. This implies, in particular, that every $\omega$-continuous domain is quasi-Polish.

A topological space $X$ is \emph{locally compact} if and only if for every $x\in X$ and open $U$ containing $x$, there is an open set $V$ and compact set $K$ such that $x\in V\subseteq K\subseteq U$. Given open sets $U$ and $V$ of a topological space $X$, we write $V\ll U$ to denote that $V$ is relatively compact in $U$ (i.e., every open cover of $U$ admits a finite subcover of $V$). As shown in \cite{Hofmann_Mislove}, a sober space $X$ is locally compact if and only if for every $x\in X$ and open $U$ containing $x$, there is open $V$ such that $x\in V \ll U$. Equivalently, a sober space is locally compact if and only if every open set is equal to the union of its relatively compact open subsets.

\begin{theorem}
Every countably based locally compact sober space is quasi-Polish.
\end{theorem}
\begin{proof}
Assume $(X,\tau)$ is a countably based locally compact sober space. Let $\{B_n\}_{n\in\omega}$ be a countable basis for $X$ that is closed under finite unions and intersections and contains the empty set. Define a function $\phi \colon X\to \cal P(\omega)$ by $\phi(x) = \{ n\in\omega \,|\, x\in B_n \}$. Clearly $\phi$ is a topological embedding of $X$ into $\cal P(\omega)$.

Let $\cal F$ be the subset of $\cal P(\omega)$ that contains exactly those $F\subseteq \omega$ satisfying:
\begin{enumerate}
\item
$(\forall m\in F) : B_m \not=\emptyset$,
\item
$(\forall m\in F)(\forall n\in\omega) : B_m \subseteq B_n \Rightarrow n \in F$,
\item
$(\forall m,n \in F)(\exists k\in F) : B_k = B_m \cap B_n$,
\item
$(\forall m\in F)(\exists n\in F) : B_n \ll B_m$.
\item
$(\forall k\in F)(\forall m,n\in \omega) : B_k = B_m\cup B_n \Rightarrow (m\in F \mbox{ or } n\in F)$,
\end{enumerate}

It is easy to check that $\cal F\in \bpi2(\cal P(\omega))$. Therefore, we only need to show that $X$ and $\cal F$ are homeomorphic. Since $\phi:X\to\cal P(\omega)$ is a topological embedding, it suffices to show that $\phi(X) = \cal F$. We first show that $\phi(X)\subseteq \cal F$. Set $x\in X$.
\begin{enumerate}
\item
If $m\in \phi(x)$ then $x\in B_m$ hence $B_m\not=\emptyset$.
\item
Assume $m\in \phi(x)$ and $n\in\omega$. If $B_m\subseteq B_n$ then $x\in B_n$ hence $n\in \phi(x)$.
\item
Assume $m,n\in \phi(x)$. Since $(B_n)_{n\in\omega}$ is closed under finite intersections, there exists $k\in\omega$ such that $B_k = B_m\cap B_n$. Clearly $x\in B_k$ hence $k\in \phi(x)$.
\item
Assume $m\in \phi(x)$. Since $X$ is locally compact, there is an open neighborhood $U$ of $x$ such that $U\ll B_m$. Let $n\in\omega$ be such that $x\in B_n \subseteq U$. Then $n\in \phi(x)$ and $B_n \ll B_m$.
\item
Assume $k\in \phi(x)$ and $m,n\in\omega$. If $B_k = B_m \cup B_n$, then either $x\in B_m$ or $x\in B_n$. Therefore, either $m\in \phi(x)$ or $n\in\phi(x)$.
\end{enumerate}
It follows that $\phi(x)\in\cal F$. Since $x\in X$ was arbitrary, $\phi(X)\subseteq \cal F$.

Next we show that $\cal F\subseteq \phi(X)$. Set $F\in \cal F$ and define
\[\cal U(F) = \{ U\in \tau \,|\, (\exists n\in F) B_n\subseteq U\}.\]
Note that:
\begin{enumerate}
\item
If $B_m \in \cal U(F)$ then $n\in F$: If $B_m\in\cal U(F)$, then there is $n\in F$ such that $B_n\subseteq B_m$, hence $m\in F$ because $F\in\cal F$. In particular, $\emptyset \not\in \cal U(F)$.
\item
$\cal U(F)$ is an upper set: If $U\in\cal U(F)$ then there is $n\in F$ such that $B_n\subseteq U$. So clearly if $U\subseteq V\in\tau$, then $B_n\subseteq V$ hence $V\in\cal U(F)$.
\item
$\cal U(F)$ is a filter: If $U,V\in\cal U(F)$, then there are $m,n\in F$ such that $B_m\subseteq U$ and $B_n\subseteq V$. Then there is $k\in F$ such that $B_k=B_m\cap B_n$ and clearly $B_k \subseteq U\cap V$. Therefore, $U\cap V\in\cal U(F)$.
\item
$\cal U(F)$ is a Scott-open filter: Assume $\cal D$ is a directed set of open subsets of $X$ and $\bigcup \cal D\in \cal U(F)$. By definition of $\cal U(F)$ there is $m\in F$ such that $B_m \subseteq \bigcup \cal D$, and by the assumptions on $F$ there is $n\in F$ with $B_n \ll B_m$. Since $B_m \subseteq \bigcup\cal D$ there is $W\in\cal D$ such that $B_n \subseteq W$. It follows that $W\in \cal U(F)$ hence $\cal D\cap \cal U(F)\not=\emptyset$.
\item
$\cal U(F)$ is a prime Scott-open filter: If $U\cup V\in\cal U(F)$, then there is $m,n\in F$ such that $B_m\subseteq U\cup V$ and $B_n\ll B_m$. Let $I_U = \{i\in\omega\,|\, B_i\subseteq U\}$ and $I_V = \{i\in\omega\,|\, B_i\subseteq V\}$. Then $B_m \subseteq U\cup V = \bigcup_{i\in I_U\cup I_V}B_i$, so there is finite $I'_U\subseteq I_U$ and finite $I'_V \subseteq I_V$ such that $B_n \subseteq \bigcup_{i\in I'_U\cup I'_V}B_i$. Since $(B_n)_{n\in\omega}$ is closed under finite unions, there are $i_U, i_V\in \omega$ such that $B_{i_U} = \bigcup_{i\in I'_U} B_i$ and $B_{i_V} = \bigcup_{i\in I'_V} B_i$. Since $B_n \subseteq B_{i_U}\cup B_{i_V}$, either $i_U\in F$ or else $i_V\in F$. If $i_U\in F$ then $U\in\cal U(F)$, otherwise $i_V\in F$ and $V\in \cal U(F)$.
\end{enumerate}

Since $X$ is sober and locally compact, it follows by results of K. Hofmann and M. Mislove (see Lemma 2.23 and Proposition 2.24 in \cite{Hofmann_Mislove}) that $\cal U(F)$ equals the set of open neighborhoods of some $x\in X$. Therefore, $F =  \{ n\in\omega \,|\, x\in B_n \} = \phi(x)$. Since $F\in\cal F$ was arbitrary, $\cal F\subseteq \phi(X)$.
\qed
\end{proof}

Every continuous domain is locally compact and sober (see Proposition III-3.7 in \cite{etal_scott}). Therefore, we immediately obtain the following.

\begin{corollary}
Every $\omega$-continuous domain is quasi-Polish.
\qed
\end{corollary}


\section{Admissible representations of quasi-Polish spaces}\label{sec:admissiblerepresentations}

In this section we characterize quasi-Polish spaces as precisely the countably based spaces that have an admissible representation with Polish domain. Equivalently, quasi-Polish spaces are precisely the countably based spaces with total admissible representations defined on all of $\baire$. Admissible representations of topological spaces are fundamental to the development of computable analysis under the Type 2 Theory of Effectivity (see \cite{weihrauch}).

\begin{definition}[K. Weihrauch \cite{weihrauch}, M. Schr\"{o}der \cite{schroder}]
A partial continuous function $\rho\colon\subseteq \baire\to X$ is an \emph{admissible representation} of $X$ if and only if for every partial continuous $f\colon\subseteq \baire\to X$ there exists a partial continuous $g\colon\subseteq \baire\to\baire$ such that $f = \rho\circ g$.
\qed
\end{definition}

A characterization of the topological spaces which have admissible representations has been given by M. Schr\"{o}der \cite{schroder}. Every space which has an admissible representation satisfies the $T_0$-axiom.

The major importance of admissible representations is due to the following fact. If $X$ and $Y$ are countably based spaces\footnote{The statement still holds for non-countably based $X$ and $Y$ if we either require $X$ and $Y$ to be sequential spaces or we relax the continuity requirement of $f$ to sequential continuity (see \cite{schroder} for details).}, and $\rho_X\colon\subseteq \baire\to X$ and $\rho_Y\colon \subseteq \baire\to Y$ are admissible representations, then a function $f\colon X\to Y$ is continuous if and only if there exists a continuous partial function $g\colon\subseteq \baire\to\baire$ such that $f\circ \rho_X = \rho_Y\circ g$. This reduces the analysis of continuous functions between represented spaces to the analysis of (partial) continuous functions on $\baire$, which are usually better understood and carry a natural definition of computability.

An example of an admissible representation is the function $\delta \colon \baire \to \cal P(\omega)$, defined as 
\[\delta(p) = \{ x\in\omega \,|\, \exists n : p(n)=x+1\}\]
for $p\in\baire$ (we have already used this function in the proof of Lemma \ref{lem:open_surj_from_baire}). This is sometimes called the \emph{enumeration representation} of $\cal P(\omega)$, and is known to be admissible (see Theorem 4.6 and Corollary 4.8 of \cite{kreitz_weihrauch} and Chapter 3 of \cite{weihrauch}). To see why $\delta$ is admissible, let $f\colon\subseteq \baire \to \cal P(\omega)$ be a partial continuous function. The set $\uparrow\!\{n\} = \{X\in\cal P(\omega) \,|\, n\in X\}$ is open for each $n\in\omega$, thus the continuity of $f$ implies that $n\in f(p)$ if and only if $f(\uparrow\!p[i]) \subseteq \uparrow\!\{n\}$ for some $i \in\omega$, where $p[i]$ denotes the initial segment of $p$ of length $i$. Let $r\colon\omega\to\omega$ be a bijection such that $r^{-1}(n)$ is infinite for each $n\in\omega$. Then $g\colon\subseteq\baire\to\baire$, defined as
\begin{eqnarray*}
g(p)(i) &=& \left\{
\begin{array}{ll}
r(i)+1 & \mbox{ if } f(\uparrow\!p[i]) \subseteq \uparrow\!\{r(i)\},\\
0 & \mbox{ otherwise }
\end{array}
\right.
\end{eqnarray*}
for $p\in dom(f)$, is continuous and satisfies $f=\delta\circ g$.

We now move on to our characterization of quasi-Polish spaces in terms of admissible representations with Polish domains. The following lemma is a simple adaption of a result by F. Hausdorff in \cite{hausdorff}.

\begin{lemma}\label{lem:open_restrictions}
Assume $X$ is a separable metric space, $Y$ is a countably based $T_0$-space, and $\phi\colon X \to Y$ is a continuous surjection. If there is $A\subseteq X$ such that $\phi$ restricted to $A$ is an open continuous surjection onto $Y$, then there is a $G_{\delta}$ subset $P$ of $X$ containing $A$ such that $\phi$ restricted to $P$ is an open continuous surjection onto $Y$.
\end{lemma}
\begin{proof}
Since $X$ is a separable metric space we can define a system of open subsets $U_{\sigma}$ ($\sigma\in\omega^{<\omega}$) of $X$ such that
\begin{enumerate}
\item
$U_{\langle \rangle} = X$, where $\langle \rangle$ is the empty sequence,
\item
$U_{\sigma}$ has diameter less than $1/|\sigma|$ when $\sigma$ is not empty,
\item
$U_{\sigma} = \bigcup_{n\in\omega}U_{\sigma\diamond n}$,
\item
$\sigma \preceq \sigma'$ implies $U_{\sigma}$ contains the closure of $U_{\sigma'}$,
\end{enumerate}
where $|\sigma|$ denotes the length of $\sigma$ and $\sigma\diamond n$ denotes the sequence obtained by appending $n$ to the end of $\sigma$.

For $\sigma\in\omega^{<\omega}$ define 
\[V_\sigma = \phi(A\cap U_{\sigma}) \qquad\mbox{and}\qquad  W_\sigma = U_\sigma \cap \phi^{-1}(V_\sigma).\]
$V_\sigma$ is open in $Y$ because $\phi$ restricted to $A$ is an open map, and $W_\sigma$ is open in $X$ because $\phi$ is continuous. Now define
\[ P = \bigcap_{n\in\omega} \bigcup_{\sigma\in\omega^n} W_\sigma.\]
Clearly $P$ is a $G_\delta$ subset of $X$. For any $x\in A$, there is some $p\in\baire$ such that $\{x\} = \bigcap_{\sigma\prec p}U_\sigma$. Then for each $\sigma\prec p$, $x\in A\cap U_\sigma$, thus $\phi(x)\in V_\sigma$, hence $x\in W_\sigma$. Therefore, $x\in P$. It follows that $A\subseteq P$.

We show that $\phi(P\cap W_\sigma)=V_\sigma$ for each $\sigma\in\omega^{<\omega}$. Clearly, $\phi(P\cap W_\sigma)\subseteq \phi(P)\cap \phi(W_\sigma) \subseteq Y\cap V_\sigma = V_\sigma$. Conversely, for any $y\in V_\sigma$, there is $x\in A\cap U_\sigma$ such that $\phi(x)=y$. Then $x\in A\cap W_\sigma$, and since $A\subseteq P$, $x\in P\cap W_\sigma$. It follows that $V_\sigma \subseteq \phi(P\cap W_\sigma)$.

By definition of $P$, for each $x\in P$ and $n\in\omega$, there is $\sigma\in\omega^n$ with $x\in W_\sigma$. When $n>0$, we conclude that $W_\sigma$ is an open neighborhood of $x$ with diameter less than $1/n$ because $W_\sigma\subseteq U_\sigma$. Therefore, the sets $P\cap W_\sigma$ ($\sigma\in\omega^{<\omega}$) form a basis for the topology on $P$. As $\phi(P\cap W_\sigma)=V_\sigma$ is open for each $\sigma\in\omega^{<\omega}$, it follows that $\phi$ restricted to $P$ is an open continuous surjection onto $Y$.
\qed
\end{proof}

\begin{theorem}
A countably based space $X$ is quasi-Polish if and only if there is an admissible representation $\rho\colon\subseteq \baire\to X$ of $X$ such that $dom(\rho)$ is Polish.
\end{theorem}
\begin{proof}
Assume $X$ is quasi-Polish. Without loss of generality, we can assume that $X\in \bpi 2(\cal P(\omega))$. We already saw that the total function $\delta \colon \baire \to \cal P(\omega)$, defined as $p\mapsto \{ x\in\omega \,|\, \exists n : p(n)=x+1\}$, is an admissible representation of $\cal P(\omega)$. Because $\delta$ is continuous, $P=\delta^{-1}(X)$ is a $\bpi 2$ subset of $\baire$, hence a Polish space. The function $\rho\colon\subseteq \baire\to X$, defined as the restriction of $\delta$ to $P$, is easily seen to be an admissible representation of $X$ with $dom(\rho)$ a Polish space.

Conversely, assume $\rho\colon\subseteq \baire\to X$ is an admissible representation of $X$ and $dom(\rho)$ is Polish. By a result of V. Brattka and P. Hertling (see Lemma 7 in \cite{brattka_hertling}), there is a subset $A$ of $dom(\rho)$ such that the restriction of $\rho$ to $A$ is an open continuous surjection onto $X$. By Lemma \ref{lem:open_restrictions}, there is a $G_\delta$ subset $P$ of $dom(\rho)$ containing $A$ such that the restriction of $\rho$ to $P$ is an open continuous surjection. $P$ is Polish because it is a $G_\delta$ subset of the Polish space $dom(\rho)$, hence $X$ is quasi-Polish by Theorem~\ref{thrm:open_surjection_preserves_pi2}.
\qed
\end{proof}

V. Brattka has shown (Corollary 4.4.12 in \cite{brattka_thesis}) that every Polish space $X$ has a total admissible representation $\rho\colon \baire\to X$. By composing representations we obtain the following.

\begin{theorem}
A countably based space $X$ is quasi-Polish if and only if there is a total admissible representation $\rho\colon \baire\to X$ of $X$.
\qed
\end{theorem}

The requirement that $X$ be countably based in the above theorems can not be dropped. In Example 3 of \cite{schroder}, an admissible representation is constructed for a countable Hausdorff space which is not first-countable (hence not quasi-metrizable). It is easy to see that the domain of the representation in this example is Polish, which implies that the space has a total admissible representation. An interesting question is whether or not the completeness properties of quasi-Polish spaces generalize in some way to all spaces with total admissible representations.


\section{A game theoretic characterization of quasi-Polish spaces}\label{sec:gamecharacterization}

In this section we give a game theoretic characterization of quasi-Polish spaces by a simple modification of the strong Choquet game (see \cite{kechris}).

\begin{definition}
Given a non-empty topological space $(X,\tau)$, the game $\cal G(X,\tau)$ is defined as follows.
\[\begin{array}{lccccc}
\mbox{Player I: }  & x_0, U_0    &     & x_1, U_1  &     & \ldots \\
\mbox{Player II: } &		& V_0 &           & V_1 & \ldots
\end{array}\]
Players I and II take turns playing non-empty open subsets of $X$ such that $U_0\supseteq V_0\supseteq U_1\supseteq \ldots$, but additionally Player I is required to play any point $x_n\in U_n$ and II must then play $V_n\subseteq U_n$ with $x_n\in V_n$.

Player II wins the game $\cal G(X,\tau)$ if and only if $\{V_i\,|\, i\in\omega\}$ is a neighborhood basis of some $x\in X$ (i.e., for any open $U\subseteq X$ containing $x$, there is $i\in\omega$ such that $x\in V_i \subseteq U$). Equivalently, Player II wins if and only if $\{U_i\,|\, i\in\omega\}$ is a neighborhood basis of some $x\in X$.
\qed
\end{definition}

If the topology of $X$ is clear from context, then we write $\cal G(X)$ instead of $\cal G(X,\tau)$. The \emph{strong Choquet game} for a topological space $X$ is played with the same rules as $\cal G(X)$, but with the exception that Player II wins if and only if $\bigcap_{n\in\omega}U_n$ is non-empty. A topological space $X$ is a \emph{strong Choquet space} if and only if Player II has a winning strategy\footnote{A precise definition for the term ``winning strategy'' can be found in \cite{kechris}.} for the strong Choquet game on $X$. It immediately follows that if Player II has a winning strategy in the game $\cal G(X)$, then $X$ is a strong Choquet space.

F. Dorais and C. Mummert \cite{dorais_mummert} have investigated the above game for $T_1$-spaces in terms of ``convergent'' strategies for Player II in the strong Choquet game. In particular, they showed that for an arbitrary $T_1$-space $X$, Player II has a winning strategy in the game $\cal G(X)$ if and only if $X$ is the open continuous image of a complete metric space. The following theorem shows that this result extends to all countably based $T_0$-spaces.

\begin{theorem}\label{thrm:game_theoretic_characterization}
If $X$ is a non-empty countably based $T_0$-space, then Player II has a winning strategy in the game $\cal G(X)$ if and only if $X$ is a quasi-Polish space.
\end{theorem}
\begin{proof}
Let $(X,d)$ be a complete quasi-metric space. We show that the following strategy is winning for Player II in the game $\cal G(X)$:

\begin{quote}
If Player I plays $(x_n, U_n)$ on the $n$-th step, then Player II responds by playing $V_n = B_d(x_n,\varepsilon_n)$, where $\varepsilon_n$ is chosen so that $0<\varepsilon_n\leq 1/n$ and $\overline{B}_d(x_n,\varepsilon_n)\subseteq U_n$.
\end{quote}

Since $x_{n+1}\in V_n$ for all $n\in\omega$, it is clear that $(x_n)_{n\in\omega}$ is a Cauchy sequence. Let $x\in X$ be the limit of $(x_n)_{n\in\omega}$ with respect to $\tau_{\widehat{d}}$.

We first show that $x\in V_n$ for all $n\in\omega$. Assume for a contradiction that $d(x_n,x) = \varepsilon_n + \varepsilon$ for some $n\in\omega$ and $\varepsilon>0$. Since $(x_n)_{n\in\omega}$ converges to $x$ with respect to to $\tau_{\widehat{d}}$, there must be some $m\geq n$ such that $d(x_m,x)<\varepsilon$. Since $x_m\in B_d(x_n,\varepsilon_n)$, it follows that $d(x_n,x)\leq d(x_n,x_m)+d(x_m,x) < \varepsilon_n + \varepsilon$, a contradiction.

To finish the proof, it suffices to show that for all $\varepsilon>0$, there is some $n\in\omega$ such that $x\in V_n \subseteq B_d(x,\varepsilon)$. Choose $n$ large enough that $d(x,x_n)<\varepsilon/2$ and $1/n < \varepsilon/2$. Then $x\in V_n$ and for all $y\in V_n$,
\[d(x,y)\leq d(x,x_n) + d(x_n,y) < \varepsilon/2 + \varepsilon/2 = \varepsilon,\]
hence $y\in B_d(x,\varepsilon)$.

For the converse, assume that $X$ is a non-empty countably based $T_0$-space, and that Player II has a winning strategy in $\cal G(X)$. Let $\cal B$ be a countable basis for $X$ which contains $X$. We can assume that Player II always plays elements of $\cal B$. Indeed, if on the $n$-th step, Player I played $(x_n,U_n)$ and Player II's strategy is to play $V_n$ in response, then Player II could revise his strategy to play any $B\in\cal B$ such that $x_n\in B \subseteq V_n$ and still have a winning strategy. Below we assume that Player II plays according to some fixed winning strategy in which he only plays elements of $\cal B$.

We next define a function $f\colon \omega^{<\omega}\to \cal B$ and simultaneously associate each $\sigma\in\omega^{<\omega}$ with a set of plays of the game $\cal G(X)$, which we will call the \emph{$\sigma$-runs}. The $\sigma$-runs will be characterized by being runs of the game in which Player I plays the open set $f(\sigma')$ on the $|\sigma'|$-th step for all $\sigma'\preceq \sigma$.

For the empty sequence $\langle\rangle$, define $f(\langle\rangle) = X$ and define the $\langle\rangle$-runs to be the set of all plays of the game $\cal G(X)$ in which Player I starts by playing $(x_0,X)$ for some $x_0\in X$.

Now assume $f(\sigma)$ has been defined and $|\sigma|= n$. Let $\cal S$ be the set of all $B\in\cal B$ such that there is some $\sigma$-run of the game $\cal G(X)$ in which on the $n$-th step Player I plays $(x_n, f(\sigma))$ (for some $x_n\in X$) and in response Player II's strategy is to play $V_n=B$. Then $\cal S$ is non-empty and countable, so let $(S_i)_{i\in\omega}$ be an enumeration (possibly with repeats) of the elements of $\cal S$. For each $i\in\omega$, define $f(\sigma\diamond i)=S_i$. Let the $\sigma\diamond i$-runs be the subset of the $\sigma$-runs in which on the $n$-th step Player II responds with $f(\sigma\diamond i)$ and in the $(n+1)$-th step Player I continues by playing $(x_{n+1},f(\sigma\diamond i))$ for some $x_{n+1}\in f(\sigma\diamond i)$. 

We next define a function $\phi\colon \baire\to X$. For each $p\in\baire$ we can associate the set of all plays of the game $\cal G(X)$ in which Player I plays $(x_n,f(p[n]))$ on the $n$-th step, where $p[n]$ is the initial prefix of $p$ of length $n$ and $x_n\in f(p[n])$. As Player II's strategy is winning, $\{f(p[n]) \,|\, n\in\omega\}$ must be a neighborhood basis of some $x_p \in X$, which is necessarily unique because $X$ is a $T_0$-space. Define $\phi(p)=x_p$.

We finish the proof of the theorem by showing that $\phi\colon \baire\to X$ is an open continuous surjection.

To see that $\phi$ is continuous, fix $p\in \baire$ and open neighborhood $U$ of $\phi(p)\in X$. Then there must be some $\sigma\prec p$ such that $\phi(p)\in f(\sigma)\subseteq U$. As $\phi(q)\in f(\sigma)$ for all $q\in\baire$ extending $\sigma$, $\phi(\uparrow \sigma)\subseteq U$. It follows that $\phi$ is continuous.

Finally, we show that $\phi$ is an open surjection by proving $\phi(\uparrow\sigma)=f(\sigma)$ for each $\sigma\in\omega^{<\omega}$ (surjectivity follows because $f(\langle\rangle)=X$). It is clear by the definition of $\phi$ that $\phi(\uparrow\sigma)\subseteq f(\sigma)$. For the converse, assume $x\in f(\sigma)$. Then there is a $\sigma$-run of the game $\cal G(X)$ in which Player I plays $(x,f(\sigma))$ in the $|\sigma|$-th step. Player I can continue this $\sigma$-run by always playing $(x,V_n)$ in response to Player II playing $V_n$ in the $n$-th step. As a result, there must be some $p\in\baire$ extending $\sigma$ that is associated with this run of the game, hence $x=\phi(p)\in\phi(\uparrow~\sigma)$.

Since $\phi\colon\baire\to X$ is an open continuous surjection, it follows from Theorem \ref{thrm:open_surjection_preserves_pi2} that $X$ is quasi-Polish.
\qed
\end{proof}

It follows that every quasi-Polish space is strong Choquet. Also, every strong Choquet space is a Baire space (i.e., countable intersections of dense open sets are dense), thus we obtain the following.

\begin{corollary}
Every quasi-Polish space is a Baire space.
\qed
\end{corollary}


\section{Embedding quasi-Polish spaces into $\omega$-continuous domains}\label{sec:domainembedding}

In this section we give a domain-theoretic characterization of quasi-Polish spaces. We then show some applications to modeling spaces as the maximal elements of a domain.

\begin{theorem}\label{thrm:domain_embedding}
The following are equivalent for a topological space $X$:
\begin{enumerate}
\item
$X$ is a quasi-Polish space,
\item
$X$ is homeomorphic to the set of non-compact elements of some $\omega$-continuous domain,
\item
$X$ is homeomorphic to the set of non-compact elements of some $\omega$-algebraic domain.
\end{enumerate}
\end{theorem}
\begin{proof}
The implication from $3$ to $2$ is trivial. To see that $2$ implies $1$, first note that if $x$ is a compact element of an $\omega$-continuous domain $D$, then the singleton $\{x\}$ is in $\bdelta 2(D)$. As there are at most a countably infinite number of compact elements in an $\omega$-continuous domain, the set of compact elements of $D$ is $\bsigma 2$, hence the set of non-compact elements is $\bpi 2$. The implication from $2$ to $1$ follows because every $\omega$-continuous domain is quasi-Polish.

It only remains to show that $1$ implies $3$. So assume without loss of generality that $X\in \bpi 2(\cal P(\omega))$. Then we have
\[\cal P(\omega)\setminus X = \bigcup_{i\in\omega} U_i\setminus V_i\]
for some appropriate choice of open sets $U_i,V_i$. We define
\[\cal F = \{ F\in\cal P(\omega) \,|\, F \mbox{ is finite and } (\exists x\in X) \colon F\subseteq x\}.\]
We partially order $\cal F\times\omega$ by $\langle F_1,n_1\rangle \sqsubseteq \langle F_2,n_2\rangle$ if and only if either
\begin{enumerate}
\item
$F_1=F_2$ and $n_1=n_2$, or else
\item
$F_1\subseteq F_2$ and $n_1 < n_2$ and $(\forall m\leq n_1) \colon F_1\in U_m \Rightarrow F_2 \in V_m$.
\end{enumerate}
It is immediate that $\sqsubseteq$ is reflexive and anti-symmetric. To see that it is transitive, simply note that if $F_2 \in V_m$ and $F_2\subseteq F_3$, then $F_3\in V_m$ because $V_m$ is open.

We let $\cal I$ denote the set of all ideals of $\langle\cal F\times\omega,\sqsubseteq\rangle$ ordered by inclusion. By Proposition I-4.10 in \cite{etal_scott}, $\cal I$ is an $\omega$-algebraic domain whose compact elements are precisely the principal ideals. We complete the proof by showing that $X$ is homeomorphic to the subspace of non-principal ideals of $\cal I$.

Define $\phi\colon X\to \cal I$ by $\phi(x)=\{\langle F,n\rangle\in \cal F\times\omega \,|\, F\subseteq x\}$. We first show that $\phi(x)$ is an ideal and thus $\phi$ is well-defined.

Clearly $\phi(x)$ is a lower set with respect to $\sqsubseteq$. To show $\phi(x)$ is directed, assume $\langle F_1,n_1\rangle,\langle F_2,n_2\rangle\in \phi(x)$. Let $n=\max\{n_1,n_2\}+1$. For all $m\leq n$ we define a finite $G_m\subseteq\omega$ as follows. If $F_1\in U_m$ or $F_2\in U_m$ then $x\in U_m$ since $U_m$ is open, hence $x\in V_m$ because $x\in X$. We can therefore choose $G_m\in \cal F$ so that $F_1\cup F_2 \subseteq G_m \subseteq x$ and $G_m\in V_m$. If on the other hand, $F_1\not\in U_m$ and $F_2\not\in U_m$, then let $G_m=F_1\cup F_2$. By defining $F=\bigcup_{m\leq n}G_m$ it follows by our construction that $\langle F,n\rangle\in\phi(x)$ and $\langle F_1,n_1\rangle,\langle F_2,n_2\rangle\sqsubseteq \langle F,n\rangle$. Therefore, $\phi(x)$ is an ideal and $\phi$ is a well-defined function.

By the above argument it is clear that $\phi(x)$ is a non-principal ideal for each $x\in X$. For the converse, assume that $I\in\cal I$ is non-principal. Let $x = \bigcup\{ F\,|\, \langle F,n\rangle \in I\}$. If $x\in U_n$ for some $n\in \omega$, then there is finite $F\subseteq x$ such that $F\in U_n$. Since $F$ is finite and $I$ is directed, there is some $\langle F_0,n_0\rangle\in I$ such that $F\subseteq F_0$. Since $I$ is not principal, there is $\langle F_1,n_1\rangle\in I$ distinct from $\langle F_0,n_0\rangle$ such that $\langle F_0,n_0\rangle \sqsubseteq \langle F_1,n_1\rangle$. In particular, $n_0<n_1$ by definition of $\sqsubseteq$. Thus, there exists a finite strictly increasing chain $\langle F_0,n_0\rangle \sqsubseteq \langle F_1,n_1\rangle\sqsubseteq \cdots \sqsubseteq \langle F_m,n_m\rangle \sqsubseteq \langle F_{m+1},n_{m+1}\rangle$ of elements in $I$ such that $n_m > n$. Clearly, $F_m\in U_n$ because $F\subseteq F_m$, and it follows that $F_{m+1}\in V_n$ by definition of $\sqsubseteq$. Since $F_{m+1}\subseteq x$, it follows that $x\in V_n$. As $n\in\omega$ was arbitrary, it follows that $x\in X$.

Since $x =\bigcup\{ F\,|\, \langle F,n\rangle \in I\}$, it is clear that $I\subseteq \phi(x)$. On the other hand, if $\langle F,n\rangle\in\phi(x)$, then by repeating the argument from the previous paragraph, there exist elements $\langle F',n'\rangle$ and $\langle F'',n''\rangle$ in $I$ satisfying $F\subseteq F'$, $n<n'$, and $\langle F',n'\rangle\sqsubseteq \langle F'',n''\rangle$. For all $m\leq n$, if $F\in U_m$ then $F'\in U_m$, hence $F''\in V_m$. It follows that $\langle F,n\rangle \sqsubseteq \langle F'',n''\rangle$, hence $\langle F,n\rangle\in I$. Therefore, $I = \phi(x)$. 

It follows that the image of $\phi$ is precisely the set of non-principal ideals, hence the non-compact elements in $\cal I$. Since the Scott-topology on $\cal I$ is generated by sets of the form $\{I\in\cal I\,|\, \langle F,n\rangle\in I\}$ for $\langle F,n\rangle\in\cal F\times \omega$, it is easily seen that $\phi$ is a homeomorphic embedding of $X$ into $\cal I$.
\qed
\end{proof}

Given a topological space $X$, let $Max(X)$ denote the set of maximal elements of $X$ with respect to the specialization order. Every quasi-Polish space is a dcpo with respect to the specialization order by virtue of being sober, hence $Max(X)$ is non-empty when $X$ is a non-empty quasi-Polish space.

In \cite{Martin_ideal}, K. Martin defines an $\omega$-ideal domain to be an $\omega$-algebraic domain in which every element is compact or maximal with respect to the specialization order. Furthermore, an $\omega$-ideal model of a topological space $X$ is defined to be an $\omega$-ideal domain $D$ in which $X$ is homeomorphic to $Max(D)$. 

If $D$ is an $\omega$-ideal domain, then $D\setminus Max(D)$ is a countable collection of compact elements, hence $Max(D)\in\bpi 2(D)$. Clearly every space that has an $\omega$-ideal model must satisfy the $T_1$-separation axiom, and conversely it is clear that the construction given in the proof of Theorem \ref{thrm:domain_embedding} is an $\omega$-ideal model when applied to a $T_1$ quasi-Polish space. We therefore obtain the following.

\begin{corollary}
A topological space has an $\omega$-ideal model if and only if it is quasi-Polish and satisfies the $T_1$-separation axiom.
\qed
\end{corollary}

K. Martin has shown that if $D$ is an $\omega$-continuous domain and $Max(D)$ is metrizable, then $Max(D)$ is $G_\delta$ in $D$. Since $D$ is quasi-Polish, this implies that $Max(D)$ is quasi-Polish, hence Polish because it is metrizable. In the proof of Theorem \ref{thrm:domain_embedding}, the maximal elements of $X$ and the $\omega$-continuous domain in which it is embedded coincide, so we obtain the following.

\begin{corollary}
If $X$ is quasi-Polish and $Max(X)$ is metrizable, then $Max(X)$ is Polish.
\qed
\end{corollary}

Things become more complicated if $X$ is not an $\omega$-ideal domain and $Max(X)$ is not metrizable. H. Bennett and D. Lutzer \cite{bennett_lutzer} provided the first example of an $\omega$-algebraic domain $X$ in which the subspace $Max(X)$ is a non-metrizable Hausdorff space. An interesting property of their construction is that $Max(X)$ contains the space of rationals as a closed subspace. Since the rationals are not completely metrizable, it is clear that $Max(X)$ is not quasi-Polish in this case.

To characterize the complexity of $Max(X)$ for arbitrary quasi-Polish $X$, we begin with a lemma. Below we let $\pi_X$ denote the projection from $X\times Y$ onto $X$.

\begin{lemma}\label{lem:analytic_equiv}
The following are equivalent for a subset $A$ of a quasi-Polish space $X$:
\begin{enumerate}
\item
$A=\pi_X(F)$ for some $\bpi 2$ subset $F\subseteq X\times \baire$.
\item
$A=\pi_X(F)$ for some quasi-Polish $Y$ and $\bpi 2$ subset $F\subseteq X\times Y$.
\item
$A=\pi_X(B)$ for some quasi-Polish $Y$ and Borel subset $B\subseteq X\times Y$.
\item
$A=f(\baire)$ for some continuous $f\colon \baire\to X$.
\item
$A=f(Y)$ for some quasi-Polish $Y$ and continuous $f\colon Y\to X$.
\end{enumerate}
\end{lemma}
\begin{proof}
The implications $1 \Rightarrow 2$ and $2 \Rightarrow 3$ are obvious. Now assuming $3$, $X\times Y$ is quasi-Polish so there is a continuous open surjection $g\colon \baire \to X\times Y$. Then $g^{-1}(B)$ is Borel in $\baire$, so there is a continuous function $h\colon \baire\to\baire$ such that $h(\baire)=g^{-1}(B)$ (see Theorem 13.7 in \cite{kechris}). Therefore, $f = \pi_X \circ g \circ h$ satisfies $4$. The implication from $4$ to $5$ is trivial, and $5$ implies $4$ by composing with a continuous surjection from $\baire$ to $Y$. Finally, assume $4$ and define $F=\{\langle x,y\rangle\in X\times \baire\,|\, f(y)=x\}$. Then $F\in\bpi 2(X\times \baire)$ and $\pi_X(F)=f(\baire)=A$, which proves $1$.
\qed
\end{proof}

This equivalence allows us to extend the definition of analytic sets to quasi-Polish spaces.

\begin{definition}
Let $X$ be quasi-Polish. A subset $A\subseteq X$ is called \emph{analytic} if and only if it satisfies one of the equivalent conditions of Lemma \ref{lem:analytic_equiv}. A subset is \emph{co-analytic} if and only if its complement is analytic. A subset is \emph{bi-analytic} if and only if it is both analytic and co-analytic. The analytic, co-analytic, and bi-analytic subsets of $X$ will be denoted $\analytic(X)$, $\coanalytic(X)$, and $\bianalytic(X)$, respectively.
\qed
\end{definition}

If $(X,d)$ is a countably based complete quasi-metric space, then $(X,\widehat{d})$ is Polish and ${\mathbf B}(X,\tau_d)={\mathbf B}(X,\tau_{\widehat{d}})$ by Theorem \ref{thrm:quasi_and_metric_borelrelation}. Therefore, most of the known properties of analytic sets in Polish spaces carry directly over to quasi-Polish spaces. For example, we have the following generalization of Souslin's Theorem.

\begin{theorem}
If $X$ is quasi-Polish, then ${\mathbf B}(X)=\bianalytic(X)$.
\qed
\end{theorem}

The above observation has already been made by D. Scott and V. Selivanov for the case of $\cal P(\omega)$.

We now give an upper bound on the complexity of the maximal elements of a quasi-Polish space.

\begin{theorem}
If $X$ is quasi-Polish then $Max(X)\in\coanalytic(X)$.
\end{theorem}
\begin{proof}
If $X$ is quasi-Polish, then $x\in Max(X)$ if and only if $(\forall y\in X)\colon x\leq y \Rightarrow x=y$, where $\leq$ is the specialization order of $X$. We have already seen that the equality relation is a $\bpi 2$ subset of $X\times X$. If we let $\{U_i\}_{i\in\omega}$ be a countable basis for $X$, then $x\leq y$ if and only if $(\forall i\in\omega)\colon x\in U_i \Rightarrow y\in U_i$, so the specialization order is also $\bpi 2$ in $X\times X$. Thus, $B=\{\langle x,y\rangle \,|\, x\leq y \Rightarrow x=y\}$ is Borel in $X\times X$, and $Max(X)$ is the complement of the projection of the complement of $B$.
\qed
\end{proof}

It turns out that this is the best lower bound possible in general. C. Mummert has shown (Theorem 2.8 in \cite{mummert}) that any co-analytic subset of $\baire$ can be embedded into a relatively closed subset of the maximal elements of an $\omega$-continuous domain. If we choose a co-analytic set that is not analytic, then the maximal elements of such a domain can not be Borel. C. Mummert and F. Stephan \cite{mummert_stephan} have shown that the spaces that are homeomorphic to the maximal elements of some $\omega$-continuous domain are precisely the countably based strong Choquet spaces that satisfy the $T_1$-separation axiom (K. Martin had previously shown that the maximal elements are strong Choquet).

\begin{corollary}
A topological space $X$ is homeomorphic to $Max(Y)$ for some quasi-Polish space $Y$ if and only if $X$ is a countably based strong Choquet space satisfying the $T_1$-axiom.
\qed
\end{corollary}


\section{Scattered spaces}\label{sec:scatteredspaces}

In this section we show that scattered countably based $T_0$-spaces are quasi-Polish, which extends the known result that scattered metrizable spaces are Polish. Non-metrizable countably based scattered spaces naturally occur in the field of inductive inference as precisely those spaces that can be identified in the limit (relative to some oracle) with an ordinal mind change bound \cite{luo_schulte,debrecht_etal_3}.

\begin{definition}
A point $x$ of a topological space is \emph{isolated} if and only if $\{x\}$ is open. If $x$ is not isolated, then it is a \emph{limit point}. A space is \emph{perfect} if all of its points are limit points.
\qed
\end{definition}

\begin{definition}[see \cite{kechris}]\label{def:accum_order}
Let $X$ be a topological space. For each ordinal $\alpha$, the $\alpha$-th derived set of $X$, denoted $X^{(\alpha)}$, is defined inductively as follows:
\begin{enumerate}
\item
$X^{(0)}=X$,
\item
$X^{(\alpha+1)}=\{x\in X\,|\, x\mbox{ is a limit point of } X^{(\alpha)}\}$,
\item
If $\alpha$ is a limit ordinal, then $X^{(\alpha)}=\bigcap_{\beta<\alpha}X^{(\beta)}$.
\end{enumerate}
The \emph{Cantor-Bendixson rank} of $X$, denoted $\abs{X}_{CB}$, is the least ordinal $\alpha$ such that $X^{(\alpha)}=X^{(\alpha+1)}$.
\qed
\end{definition}

Since $\{X^{(\alpha)}\}_{\alpha<\omega_1}$ is a decreasing sequence of closed subsets of $X$, if $X$ is countably based then $\abs{X}_{CB}$ is strictly less than $\omega_1$ (see Theorem I.6.9 in \cite{kechris}). Note that $X^{\abs{X}_{CB}}$ is a closed perfect subset of $X$.

\begin{definition}
A topological space $X$ is \emph{scattered} if and only if $X^{\abs{X}_{CB}}$ is empty.
\qed
\end{definition}

Equivalently, $X$ is scattered if and only if every subspace of $X$ contains an isolated point. It is not difficult to see that if $X$ is countably based and scattered then $X$ has at most countably many points.

The following is a separation axiom proposed by C. E. Aull and W. J. Thron \cite{aull_thron} that is strictly between the $T_0$ and $T_1$ axioms. Recall that a subset of a topological space is \emph{locally closed} if and only if it is equal to the intersection of an open set and a closed set.

\begin{definition}
A topological space $X$ satisfies the $T_D$-separation axiom if and only if $\{x\}$ is locally closed for every $x\in X$.
\qed
\end{definition}

Clearly, every scattered space satisfies the $T_D$-axiom, although the converse does not hold in general.

\begin{theorem}
A countably based space is scattered if and only if it is a countable quasi-Polish space satisfying the $T_D$-axiom.
\end{theorem}
\begin{proof}
Assume $X$ is countably based and scattered. We have already observed that $X$ is countable and satisfies the $T_D$-axiom. Consider the following strategy for Player II in the game $\cal G(X)$:
\begin{quote}
Assume Player I plays $(x_n,U_n)$ on the $n$-th step. Let $\alpha_n$ be the least ordinal such that $x_n\not\in X^{(\alpha_n+1)}$. Player II responds with open $V_n\subseteq U_n$ such that $V_n\cap X^{\alpha_n}=\{x_n\}$ and $V_n\subseteq\bigcap_{m\leq n}\{B_m\,|\, x_n\in B_m\}$, where $\{B_m\}_{m\in\omega}$ is some fixed enumeration of a basis for $X$.
\end{quote}
By the choice of $V_n$, if $x_{n+1}\not=x_n$ then $x_{n+1}\not\in X^{(\alpha_n)}$, hence $\alpha_{n+1}<\alpha$. It follows that there is some $n_0\in\omega$ such that $x_n=x_{n_0}$ for all $n\geq n_0$. Clearly, Player II's strategy enumerates a neighborhood basis for $x_{n_0}$.

For the converse, assume $X$ is a countable quasi-Polish space satisfying the $T_D$-axiom and let $S\subseteq X$ be given. Both $S$ and $X\setminus S$ are countable unions of locally closed sets, hence $S\in\bdelta 2(X)$ and $S$ is quasi-Polish. Let $Cl(\cdot)$ be the closure operator on $X$. Then $S=\left( \bigcup_{x\in S} Cl(x)\right)\cap S$, so since $S$ is a Baire space there is some $x\in S$ such that $Cl(x)\cap S$ has non-empty interior (relatively in $S$). This implies there is open $U\subseteq X$ such that $x\in U\cap S\subseteq Cl(x)$. Using the $T_D$-axiom, there is open $V\subseteq X$ such that $V\cap Cl(x)=\{x\}$. Then $U\cap V \cap S = \{x\}$, hence $x$ is isolated in $S$.
\qed
\end{proof}

\begin{corollary}
Every non-empty perfect quasi-Polish space satisfying the $T_D$-axiom has cardinality $2^{\aleph_0}$.
\end{corollary}

The $T_D$-axiom is necessary in the above corollary, because the ordinal $\omega+1$ with the Scott-topology is a countable perfect quasi-Polish space which does not satisfy the $T_D$-axiom.

The following are some typical examples of countable perfect $T_0$-spaces which are not quasi-Polish:
\begin{enumerate}
\item
The rational numbers with the subspace topology inherited from the space of reals.
\item
The natural numbers with the cofinite topology. A subset of this space is open if and only if it is either empty or cofinite.
\item
The natural numbers with the Scott-topology under the usual ordering. A subset of this space is open if and only if it is empty or else of the form $\uparrow\!n = \{ m\in\omega \,|\, n \leq m\}$ for some $n\in\omega$.
\item
The rationals with the Scott-topology under the usual ordering. A subset of this space is open if and only if it is empty or else of the form $\uparrow\!r = \{ q\in\bb Q \,|\, r < q \}$ for some $r\in\bb Q$.
\end{enumerate}

Examples 1, 2, and 3, respectively satisfy the $T_2$, $T_1$, and $T_D$ separation axioms. Example 4 does not satisfy the $T_D$-axiom because no singleton subset is locally closed. It may be worthy to note, however, that Example 3 can be embedded into Example 4 as a $\bpi 2$-subset.


\section{Generalized Hausdorff-Kuratowski Theorem}\label{sec:HausdorffKuratowski}

In this section, we show that all levels of the difference hierarchy on countably based $T_0$-spaces are preserved under admissible representations. This result is then used to prove a generalization of the Hausdorff-Kuratowski Theorem for quasi-Polish spaces. The difference hierarchy on Polish spaces is well understood \cite{kechris}, and recently V. Selivanov \cite{selivanov2008} has extended many of these results to $\omega$-continuous domains.

\begin{definition}
Any ordinal $\alpha$ can be expressed as $\alpha=\beta+n$, where $\beta$ is a limit ordinal or $0$, and $n<\omega$. We say that $\alpha$ is \emph{even} if $n$ is even, and \emph{odd}, otherwise. For any ordinal $\alpha$, let $r(\alpha)=0$ if $\alpha$ is even, and $r(\alpha)=1$, otherwise. For any ordinal $\alpha$, define
\[D_{\alpha}(\{A_{\beta}\}_{\beta<\alpha})=\bigcup \{A_{\beta}\setminus\bigcup_{\gamma<\beta}A_{\gamma} \,|\, \beta<\alpha,\, r(\beta)\not=r(\alpha)\},\]
where $\{A_{\beta}\}_{\beta<\alpha}$ is a sequence of sets such that $A_{\gamma}\subseteq A_{\beta}$ for all $\gamma < \beta < \alpha$. 

For any topological space $X$ and countable ordinals $\alpha$ and $\beta$, define $D_{\alpha}(\bsigma \beta(X))$ to be the class of all sets $D_{\alpha}(\{A_{\gamma}\}_{\gamma<\alpha})$, where $\{A_{\gamma}\}_{\gamma<\alpha}$ is an increasing sequence of elements of $\bsigma \beta(X)$.
\qed
\end{definition}

The proof of the following theorem depends on a result by J. Saint Raymond (Lemma 17 in \cite{saint_raymond}) that is closely related to the Vaught transform \cite{vaught1974}. We refer the reader to Section 8 of \cite{kechris} for notions of Baire Category. Very briefly, a subset of a space is \emph{nowhere dense} if and only if its closure has empty interior. A subset is \emph{meager} if it is equal to the countable union of nowhere dense sets. A fundamental property of Polish spaces is that every non-empty open subset is non-meager.

\begin{theorem}\label{thrm:preservation_of_diffhier}
Let $X$ be a countably based $T_0$ space and $\rho\colon\subseteq \baire \to X$ an admissible representation. For any countable ordinals $\alpha,\theta>0$ and $S\subseteq X$,
\[ S\in D_{\alpha}(\bsigma \theta(X)) \iff \rho^{-1}(S)\in D_{\alpha}(\bsigma \theta(dom(\rho))).\]
\end{theorem}
\begin{proof}
One direction is trivial because $\rho$ is continuous.

For the non-trivial direction, let $\delta\colon \subseteq \baire\to X$ be an admissible representation that is an open function and has Polish fibers (such a $\delta$ necessarily exists because $X$ is countably based and $T_0$). Let $f\colon \subseteq \baire\to\baire$ be continuous such that $\delta=\rho\circ f$, which exists because $\rho$ is admissible. If $\rho^{-1}(S)\in D_{\alpha}(\bsigma \theta(dom(\rho)))$, then $\delta^{-1}(S)=f^{-1}(\rho^{-1}(S))\in D_{\alpha}(\bsigma \theta(dom(\delta)))$ because $f$ is continuous. Thus, it suffices to prove the theorem for $\delta$ instead of $\rho$.

Assume $\delta^{-1}(S)=D_{\alpha}(\{A_{\beta}\}_{\beta<\alpha})$, where $A_\beta\in \bsigma \theta(dom(\delta))$. Define 
\[B_\beta = \{ x\in X \,|\, A_\beta\cap \delta^{-1}(x) \mbox{ is non-meager in } \delta^{-1}(x)\}.\]
By the proof of Lemma 17 in \cite{saint_raymond}\footnote{Saint-Raymond states his lemma for metrizable spaces, but it is easy to generalize it to our case.} we have $B_\beta \in \bsigma \theta(X)$ for all $\beta$. We claim that $S=D_{\alpha}(\{B_{\beta}\}_{\beta<\alpha})$, which proves $S\in D_{\alpha}(\bsigma \theta(X))$.

First we show $S\subseteq D_{\alpha}(\{B_{\beta}\}_{\beta<\alpha})$. Let $x\in S$ be given. Since
\[\delta^{-1}(x)= \bigcup_{\beta<\alpha} A_{\beta}\cap \delta^{-1}(x),\]
there must be some $\beta$ such that $x\in B_\beta$. Otherwise, $\delta^{-1}(x)$ would equal the countable union of meager sets, and hence be meager in itself, which is impossible because $\delta^{-1}(x)$ is Polish. So let $\beta_x$ be the least ordinal such that $x\in B_{\beta_x}$. By choice of $\beta_x$, we have that $A_{\gamma}\cap \delta^{-1}(x)$ is meager in $\delta^{-1}(x)$ for all $\gamma<\beta_x$, hence $\bigcup_{\gamma<\beta_x} (A_{\gamma}\cap \delta^{-1}(x))$ is meager in $\delta^{-1}(x)$. Since $A_{\beta_x}\cap \delta^{-1}(x)$ is non-meager in $\delta^{-1}(x)$, it follows that $(A_{\beta_x}\setminus \bigcup_{\gamma<\beta_x} A_{\gamma})\cap \delta^{-1}(x)$ is non-meager in $\delta^{-1}(x)$ (otherwise $A_{\beta_x}\cap \delta^{-1}(x)$ would equal the union of two meager sets and be meager, a contradiction). Thus, in particular, $(A_{\beta_x}\setminus \bigcup_{\gamma<\beta_x} A_{\gamma})\cap \delta^{-1}(x)$ is non-empty, so choose $p\in (A_{\beta_x}\setminus \bigcup_{\gamma<\beta_x} A_{\gamma})\cap \delta^{-1}(x)$. Since $p \in D_{\alpha}(\{A_{\beta}\}_{\beta<\alpha})$ by hypothesis, we must have $r(\beta_x)\not=r(\alpha)$. Therefore, since $r(\beta_x)\not=r(\alpha)$ and $x\in B_{\beta_x}\setminus \bigcup_{\gamma<\beta_x}B_\gamma$, we have $x\in D_{\alpha}(\{B_{\beta}\}_{\beta<\alpha})$.

Next we show $D_{\alpha}(\{B_{\beta}\}_{\beta<\alpha}) \subseteq S$. Assume $y\in D_{\alpha}(\{B_{\beta}\}_{\beta<\alpha})$, then there is $\beta_y<\alpha$ such that $r(\beta_y)\not=r(\alpha)$ and $y\in B_{\beta_y}\setminus \bigcup_{\gamma<\beta_y}B_\gamma$. Then $A_{\beta_y}\cap\delta^{-1}(y)$ is non-meager in $\delta^{-1}(y)$, and $\bigcup_{\gamma<\beta_y} A_{\gamma}\cap\delta^{-1}(y)$ is meager in $\delta^{-1}(y)$. Then $(A_{\beta_y}\setminus \bigcup_{\gamma<\beta_y}A_\gamma)\cap \delta^{-1}(y)$ is non-meager in $\delta^{-1}(y)$, and in particular it is non-empty. Since $r(\beta_y)\not=r(\alpha)$, this implies that there is $p\in \delta^{-1}(y)$ such that $p\in D_{\alpha}(\{A_{\beta}\}_{\beta<\alpha})$. Therefore, $p\in\delta^{-1}(S)$, and it follows that $y\in S$.

Therefore, $S=D_{\alpha}(\{B_{\beta}\}_{\beta<\alpha})\in D_{\alpha}(\bsigma \theta(X))$.
\qed
\end{proof}

Since $D_1(\bsigma \alpha(X)) = \bsigma \alpha(X)$, we obtain the following result from \cite{debrecht_etal_4}.

\begin{corollary}\label{cor:preservation_of_borel_hierarchy}
Let $X$ be a countably based $T_0$-space and $\rho\colon\subseteq \baire \to X$ an admissible representation. For any $S\subseteq X$ and $1\leq \alpha<\omega_1$, $S\in \bsigma \alpha(X)$ if and only if $\rho^{-1}(S)\in \bsigma \alpha(dom(\rho))$.
\qed
\end{corollary}

The following is a generalization of the Hausdorff-Kuratowski Theorem. The case for $\theta=1$ was proven by V. Selivanov \cite{selivanov2008} for all $\omega$-continuous domains, but $\theta >1$ was left open.

\begin{theorem}\label{thrm:Haus_Kura}
If $X$ is a quasi-Polish space and $1\leq \theta<\omega_1$, then
\[\bdelta {\theta+1}(X) = \bigcup_{1\leq \alpha<\omega_1} D_{\alpha}(\bsigma \theta(X)).\]
\end{theorem}
\begin{proof}
$\bigcup_{1\leq \alpha<\omega_1} D_{\alpha}(\bsigma \theta(X))\subseteq \bdelta {\theta+1}(X)$ holds for any topological space.

Assume $S\in \bdelta {\theta+1}(X)$. Since $X$ is quasi-Polish, there exists an admissible representation $\rho\colon \subseteq \baire\to X$ of $X$ such that $dom(\rho)$ is a Polish space. Since $\rho$ is continuous, $\rho^{-1}(S)\in \bdelta {\theta+1}(dom(\rho))$. Since $dom(\rho)$ is Polish, by the Hausdorff-Kuratowski theorem (Theorem 22.27 in \cite{kechris}) there is $\alpha<\omega_1$ such that $\rho^{-1}(S)\in D_{\alpha}(\bsigma \theta(dom(\rho)))$. By the previous theorem we have $S\in D_{\alpha}(\bsigma \theta(X))$.
\qed
\end{proof}


\section{Extending quasi-Polish topologies}\label{sec:topextension}

In this section we show that classic results concerning the extension of Polish topologies naturally generalize to the quasi-Polish case. An important new result is that any (separable) metrizable extension of a quasi-Polish topology by $\bsigma 2$-sets results in a Polish topology. As corollaries, we obtain that the metric topology induced by an arbitrary (compatible) quasi-metric on a quasi-Polish space is Polish, and that the Lawson topology on an $\omega$-continuous domain is Polish.

\begin{lemma}\label{lem:delta2_extension}
Let $X$ be a quasi-Polish space and $B\in\bdelta 2(X)$. Then the topology on $X$ generated by adding $B$ as an open set is quasi-Polish.
\end{lemma}
\begin{proof}
Let $\tau$ be the original topology on $X$, and $\tau'$ the topology generated by adding $B$. By Theorem \ref{thrm:Haus_Kura}, $X\setminus B\in D_{\alpha}(\bsigma 1(X,\tau))$ for some $\alpha<\omega_1$, so let $\{A_\beta\}_{\beta<\alpha}$ be an increasing sequence of open subsets of $(X,\tau)$ such that 
\[X\setminus B = \bigcup \{A_{\beta}\setminus\bigcup_{\gamma<\beta}A_{\gamma} \,|\, \beta<\alpha,\, r(\beta)\not=r(\alpha)\}.\]

By Theorem \ref{thrm:game_theoretic_characterization}, it is sufficient to show that the following strategy is winning in the game $\cal G(X,\tau')$:
\begin{quote}
Assume Player I plays $(x_n,U_n)$ on the $n$-th step. If $x_n\in B$, then Player II plays $U_n\cap B$, and continues the rest of the game according to a winning strategy for $\cal G(U_n\cap B)$, which exists because $U_n\cap B$ is quasi-Polish.

\vspace{\baselineskip}

Otherwise, $x_m\not\in B$ for all $m\leq n$. Let $\beta_n<\alpha$ be the least ordinal such that $x_n\in A_{\beta_n} \setminus\bigcup_{\gamma<\beta_n}A_{\gamma}$ and $r(\beta_n)\not=r(\alpha)$. Since $x_n\not\in B$, there is $U'_n\in\tau$ such that $x_n\in U'_n\subseteq U_n\cap A_{\beta_n}$. Player II then plays $V_n$, where $V_n$ is determined from some fixed winning strategy for the game $\cal G(X,\tau)$ in response to Player I having played $(x_n,U'_n)$ on the $n$-th move.

\end{quote}
If $x_n\in B$ for any $n\in\omega$, then it is clear that the above strategy is winning for Player II. So we can assume that $x_n\not\in B$ for all $n\in\omega$. Since Player II is playing according to a winning strategy for $\cal G(X,\tau)$, $\{V_n\}_{n\in\omega}$ is a neighborhood basis (with respect to $\tau$) for some $x\in X$.

Clearly, $\beta_n\geq \beta_{n+1}$ for all $n\in\omega$, so there is some $m\in\omega$ such that $x_n\in A_{\beta_m}\setminus \bigcup_{\gamma<\beta_m}A_{\gamma}$ for all $n\geq m$. Since $x\in A_{\beta_m}$, if $x\in B$ then $\bigcup_{\gamma<\beta_m}A_{\gamma}$ would be an open neighborhood (with respect to $\tau$) of $x$, which is impossible because $V_n\not\subseteq \bigcup_{\gamma<\beta_m}A_{\gamma}$ for all $n\in\omega$. Therefore, $x\not\in B$ and it follows that $\{V_n\}_{n\in\omega}$ is a neighborhood basis for $x$ with respect to $\tau'$.
\qed
\end{proof}

If $X$ is Polish and $B\subseteq X$ is closed, then the topology on $X$ generated by adding $B$ as an open set is also Polish (see Lemma 13.2 in \cite{kechris}). However, if $B\in\bdelta 2(X)$ is not closed then the resulting topology might fail to be metrizable. For a simple example, let $\bb R$ be the real numbers with the usual topology, $K=\{1/n \,|\,n\in\omega, n\geq 1\}$, and $B=\bb R\setminus K$. Then $K\in \bdelta 2(\bb R)$ because it is countable and Polish, hence $B\in\bdelta 2(\bb R)$. The topology on $\bb R$ generated by adding $B$ as an open set, sometimes called the $K$-topology on $\bb R$, is quasi-Polish by Lemma \ref{lem:delta2_extension} and clearly Hausdorff, but it is not regular, hence not Polish, because $0$ and the closed set $K$ do not have disjoint neighborhoods.

\begin{lemma}\label{lem:quasipolishsequence_extension}
Let $(X,\tau)$ be a quasi-Polish space and $\{\tau_n\}_{n\in\omega}$ a sequence of quasi-Polish topologies on $X$ with $\tau\subseteq \tau_n$ for each $n\in\omega$. Then the topology on $X$ generated by $\bigcup_{n\in\omega}\tau_n$ is quasi-Polish.
\end{lemma}
\begin{proof}
The proof is a simple modification of the proof of Lemma 13.3 in \cite{kechris}. Let $\tau_{\infty}$ be the topology generated by $\bigcup_{n\in\omega}\tau_n$. Since $(X,\tau_n)$ is quasi-Polish for each $n\in\omega$, their topological product $\prod_{n\in\omega}(X,\tau_n)$ is quasi-Polish by Corollary~\ref{cor:countable_products}.

Define $f\colon (X,\tau_\infty)\to \prod_{n\in\omega}(X,\tau_n)$ as $f(x)=\langle x,x,\ldots\rangle$. If we fix a countable basis $\{U_k\}_{k\in\omega}$ for $(X,\tau)$, then it is clear that $\langle x_0,x_1,\ldots \rangle\in f(X)$ if and only if $\forall i,j,k\in\omega\colon x_i\in U_k \iff x_j\in U_k$. Therefore, $f(X)$ is $\bpi 2$ in $\prod_{n\in\omega} (X,\tau_n)$.

Since $f$ is a homeomorphism of $(X,\tau_{\infty})$ with $f(X)$, it follows that $(X,\tau_{\infty})$ is quasi-Polish.
\qed
\end{proof}

From Lemmas \ref{lem:delta2_extension} and \ref{lem:quasipolishsequence_extension} we immediately obtain the following.

\begin{theorem}\label{thrm:delta2sequence_extension}
Let $X$ be a quasi-Polish space and $A_n\in\bdelta 2(X)$ for $n\in\omega$. Then the topology on $X$ generated by adding $\{A_n\}_{n\in\omega}$ as open sets is quasi-Polish.
\qed
\end{theorem}

We also easily obtain the following generalization of a theorem by K. Kuratowski (see Theorem 22.18 in \cite{kechris}).

\begin{theorem}\label{thrm:borels_as_opens}
Let $(X,\tau)$ be quasi-Polish and $A_n\in\bsigma \alpha(X,\tau)$ for $n\in\omega$. Then there is a quasi-Polish topology $\tau'\subseteq \bsigma \alpha(X,\tau)$ extending $\tau$ such that $A_n$ is open in $(X,\tau')$ for all $n\in\omega$.
\end{theorem}
\begin{proof}
By Lemma \ref{lem:quasipolishsequence_extension} it suffices to prove the theorem for a single set $A\in\bsigma \alpha(X,\tau)$. The theorem is trivial if $\alpha = 1$, so assume $1<\alpha<\omega_1$. Then
\[ A = \bigcup_{i\in\omega}B_i\setminus B'_i\]
with $B_i,B'_i\in \bsigma{\alpha_i}(X,\tau)$ and $\alpha_i<\alpha$. By the induction hypothesis, for each $i\in\omega$ there is a quasi-Polish topology $\tau_i\subseteq \bsigma{\alpha_i}(X,\tau)$ extending $\tau$ such that $B_i,B'_i$ are open in $(X,\tau_i)$. Let $\tau'_i$ be the topology generated by adding $B_i\setminus B'_i$ to $\tau_i$. Since $B_i\setminus B'_i\in \bdelta 2(X,\tau_i)$, Lemma \ref{lem:delta2_extension} implies that $\tau'_i$ is quasi-Polish, and clearly $\tau'_i\subseteq \bsigma{\alpha_i+1}(X,\tau)\subseteq \bsigma{\alpha}(X,\tau)$. The topology $\tau'$ generated by $\bigcup_{i\in\omega}\tau'_i$ satisfies the claims of the theorem.
\qed
\end{proof}

At a glance, Theorem~\ref{thrm:borels_as_opens} may appear more general than Theorem \ref{thrm:delta2sequence_extension}. However, there is the important difference that the extending topology in Theorem~\ref{thrm:delta2sequence_extension} is generated only by the sets $A_n$, whereas the extending topology in Theorem~\ref{thrm:borels_as_opens} includes many sets other than those generated by the $A_n$ in order to make the topology quasi-Polish. 

We conclude with an important result concerning metrizable extensions of quasi-Polish topologies.

\begin{theorem}\label{thrm:sigma2_metric_polish}
Let $\tau$ and $\tau'$ be topologies on $X$ such that $(X,\tau)$ is quasi-Polish, $(X,\tau')$ is separable metrizable, and $\tau\subseteq \tau'\subseteq \bsigma 2(X,\tau)$. Then $(X,\tau')$ is Polish.
\end{theorem}
\begin{proof}
To avoid arguments involving multiple topologies on a single set, we will instead prove the following equivalent statement:

\begin{quote}
Let $X$ be a quasi-Polish space, $Y$ a separable metrizable space, and $f\colon X\to Y$ a $\bsigma 2$-measurable bijection such that $f^{-1}$ is continuous. Then $Y$ is Polish.
\end{quote}

Assume for a contradiction that $Y$ is not Polish. Using an argument similar to the first half of the proof of Theorem \ref{thrm:bicomplete_characterization}, it is immediate that $Y$ is a $\bpi 3$-subset of its completion. By a theorem of W. Hurewicz (see Theorem 21.18 in \cite{kechris}), it follows that there is a closed set $C\subseteq Y$ homeomorphic to the space of rational numbers. Since $C$ is closed, $A=f^{-1}(C)$ is in $\bpi 2(X)$, hence $A$ is quasi-Polish. 

We now let $g\colon A\to C$ be the restriction of $f$ to $A$ and consider $A$ and $C$ in their relative topologies. Let $\{U_i\}_{i\in\omega}$ be a countable basis for $C$ consisting of only clopen sets, and let $B_i=g^{-1}(U_i)$ for $i\in\omega$. Clearly each $B_i$ is in $\bdelta 2(A)$ because each $U_i$ is clopen and $g$ is $\bsigma 2$-measurable. Since $g^{-1}$ is a continuous bijection, the topology on $A$ generated by adding each $B_i$ as an open set makes $A$ homeomorphic to $C$. But Theorem \ref{thrm:delta2sequence_extension} implies that $C$ is quasi-Polish, which contradicts $C$ being homeomorphic to the rationals.
\qed
\end{proof}

\begin{corollary}
If $X$ is quasi-Polish and $d$ is any quasi-metric compatible with the topology on $X$, then $(X,\widehat{d})$ is Polish.
\qed
\end{corollary}

Note that the above corollary \emph{does not} claim that $(X,\widehat{d})$ is a complete metric space, which is false in general. It only means that the topology on $(X,\widehat{d})$ is compatible with some complete metric, possibly different than $\widehat{d}$.

For another simple application of Theorem \ref{thrm:sigma2_metric_polish}, let $X$ be an $\omega$-continuous domain and let $\tau$ be the Scott-topology on $X$. Let $\{B_i\}_{i\in\omega}$ be an enumeration of all subsets of $X$ of the form $\wayabovearrow b_0 \setminus (\uparrow\! b_1\cup\cdots\cup\uparrow\! b_n)$, where $b_0,b_1,\ldots, b_n$ are elements of some fixed countable basis (in the domain theoretic sense) for $X$. The topology $\lambda$ generated by $\{B_i\}_{i\in\omega}$ is called the \emph{Lawson topology} on $X$, and is known to be separable and metrizable for $\omega$-continuous domains (see Theorem III-4.5 and Corollary III-4.6 in \cite{etal_scott}). Since $\wayabovearrow b$ is open and $\uparrow\!b$ is $G_\delta$ with respect to the Scott-topology, it is clear that $\tau\subseteq \lambda\subseteq \bsigma 2(X,\tau)$. Theorem \ref{thrm:sigma2_metric_polish} therefore provides an alternative proof of the known fact that the Lawson topology on an $\omega$-continuous domain is Polish (compare with the proof of Proposition V-5.17 in \cite{etal_scott}).


\section{Conclusions}\label{sec:conclusion}

We have seen that the quasi-Polish spaces provide a nice common ground for the development of descriptive set theory for both Polish spaces and $\omega$-continuous domains. Our results also suggest that much can be gained by a further integration of the fields of descriptive set theory, domain theory, and generalized metrics.

It turns out that the category of quasi-Polish spaces and continuous functions has a very natural description: up to equivalence, it is the smallest full subcategory of the category of topological spaces and continuous functions which contains the Sierpinski space and is closed under countable limits. Closure under countable limits follows from our results showing that quasi-Polish spaces are closed under equalizers and countable products (Corollaries \ref{cor:equalizer} and \ref{cor:countable_products}). On the other hand, to see that every quasi-Polish space can be obtained from the Sierpinski space using countable limits, first note that $\cal P(\omega)$ is homeomorphic to the product of countably infinite many copies of the Sierpinski space. We then only need to apply a result due to D. Scott \cite{scott_datatypes} which shows that every $\bpi 2$-subset of $\cal P(\omega)$ can be obtained as the equalizer of a pair of continuous functions on $\cal P(\omega)$. In general, for any topologial space $X$ and $A\in\bpi 2(X)$, there are sequences $(U_i)_{i\in\omega}$ and $(V_i)_{i\in\omega}$ of open subsets of $X$ such that
\[x\in A \mbox{ if and only if } (\forall i\in\omega)\colon x\in U_i \Rightarrow x\in V_i.\]
Then the continuous functions $f,g \colon X\to \cal P(\omega)$, defined as
\begin{eqnarray*}
f(x) &=& \{i\in\omega \,|\, x\in U_i\},\\
g(x) &=& \{i\in\omega \,|\, x\in U_i\cap V_i\},
\end{eqnarray*}
together satisfy $x\in A$ if and only if $f(x)=g(x)$. Therefore, $A$ is the equalizer of a pair of continuous functions into $\cal P(\omega)$.

Although quasi-Polish spaces are closed under countable co-products, they are not closed under countable co-limits in general, and the category of quasi-Polish spaces is not cartesian closed. 

Recently, Reinhold Heckmann \cite{heckmann} has found a very nice proof that countably presentable locales are spatial using a generalization of the Baire category theorem. An interesting consequence of his proof is that the class of sober spaces corresponding to countably presentable locales is precisely the class of quasi-Polish spaces.

The fact that quasi-metrics do not always have a completion leaves a rather unsatisfactory gap in the theory of quasi-Polish spaces. Currently we do not know whether or not there is some alternative characterization of quasi-Polish spaces in terms of a ``complete'' generalized metric with better completion properties.

Another important task is to see how well our results can be extended to non-countably based quasi-metric spaces. For example, we do not know if a completely quasi-metrizable subspace of a non-countably based quasi-metric space is necessarily $\bpi 2$ (although Theorem \ref{thrm:bpi2_subspace_complete} shows that the converse holds). Note that the singleton set $\{\omega_1\}$ is not Borel in $(\omega_1+1)$ with the Scott-topology, so there exist non-quasi-metrizable $T_0$-spaces which contain completely metrizable non-Borel subspaces. Another interesting question is whether or not the game theoretic characterization of quasi-Polish spaces given in Section \ref{sec:gamecharacterization} applies to all completely quasi-metrizable spaces.

Finally, Theorem \ref{thrm:preservation_of_diffhier} and Corollary \ref{cor:preservation_of_borel_hierarchy} show that the Borel complexity of a subset of an admissibly represented countably based space is precisely the Borel complexity (relative to the domain of the representation) of the set of elements of $\baire$ representing the subset. This provides additional evidence that the modification of the Borel hierarchy that we have adopted in this paper is the ``correct'' definition for generalizing descriptive set theory to all countably based $T_0$-spaces. This equivalence between the complexity of subsets and their representations can serve as a basic guideline for further extending the techniques of descriptive set theory to the entire class of admissibly represented spaces. This is an important task because the admissibly represented spaces form a cartesian closed category \cite{schroder}. An important first step in this direction is to determine whether or not Corollary \ref{cor:preservation_of_borel_hierarchy} can be extended in a natural way to all admissibly represented spaces.

\section*{Acknowledgments}
I would like to thank V. Selivanov, H.-P. A. K\"{u}nzi, V. Brattka, and D. Spreen for valuable comments on earlier versions of this paper and for useful references. I am also indebted to the Kyoto Computable Analysis group, in particular H. Tsuiki, and also to K. Keimel for helping me translate reference \cite{hausdorff} during his stay in Kyoto. Finally, I would like to thank the anonymous referees for comments and suggestions which have improved this paper.

\bibliographystyle{amsplain}
\bibliography{myrefs}

\end{document}